\documentclass[
final
]{dmtcs-episciences}

\usepackage[utf8]{inputenc}
\usepackage[round]{natbib}

\usepackage{amsmath}
\usepackage{cleveref}
\usepackage{tablefootnote}
\usepackage{subfigure}

\usepackage{diagbox}
\usepackage{tikz}
\usetikzlibrary{graphs,shapes,graphs.standard}

\newcommand{\ru}[1]{\textbf{R#1}}

\newtheorem{theorem}{Theorem}
\newtheorem{question}[theorem]{Question}

\newtheorem{proposition}[theorem]{Proposition}
\crefname{proposition}{Proposition}{Propositions}
\newtheorem{lemma}[theorem]{Lemma}
\newtheorem{conjecture}[theorem]{Conjecture}

\newtheorem{observation}[theorem]{Observation}

\crefname{claim}{claim}{claims}
\Crefname{claim}{Claim}{Claims}

\DeclareMathOperator{\ad}{ad}
\DeclareMathOperator{\mad}{mad}

\title[$2$-distance $4$-coloring of planar subcubic graphs]{$2$-distance $4$-coloring of planar subcubic graphs with girth at least 21}

\author{Hoang La\affiliationmark{1,2}\thanks{hoang.la.research@gmail.com}\and Mickael Montassier\affiliationmark{1}\thanks{mickael.montassier@lirmm.fr}}
\affiliation{
LIRMM, Universit\'e de Montpellier, CNRS, Montpellier, France\\
LISN, Universit\'e Paris-Saclay, CNRS, Gif-sur-Yvette, France}

\keywords{2-distance coloring, planar graphs, discharging method}

\begin{document}

\publicationdata{vol. 26:3}{2024}{19}{10.46298/dmtcs.7563}{2021-06-09; 2021-06-09; 2024-06-12}{2024-11-01}

\maketitle

\begin{abstract}
A $2$-distance $k$-coloring of a graph is a proper vertex $k$-coloring where vertices at distance at most 2 cannot share the same color. We prove the existence of a $2$-distance $4$-coloring for planar subcubic graphs with girth at least 21. We also show a construction of a planar subcubic graph of girth 11 that is not $2$-distance $4$-colorable.
\end{abstract}

\section{Introduction}

A \emph{$k$-coloring} of the vertices of a graph $G=(V,E)$ is a map $\phi:V \rightarrow\{1,2,\dots,k\}$. A $k$-coloring $\phi$ is a \emph{proper coloring}, if and only if, for each edge $xy\in E,\phi(x)\neq\phi(y)$. In other words, no two adjacent vertices share the same color. The \emph{chromatic number} of $G$, denoted by $\chi(G)$, is the smallest integer $k$ such that $G$ has a proper $k$-coloring.  A generalization of $k$-coloring is $k$-list-coloring.
A graph $G$ is {\em $L$-list colorable} if for a
given list assignment $L=\{L(v): v\in V(G)\}$ there is a proper
coloring $\phi$ of $G$ such that for all $v \in V(G), \phi(v)\in
L(v)$. If $G$ is $L$-list colorable for every list assignment $L$ with $|L(v)|\ge k$ for all $v\in V(G)$, then $G$ is said to be {\em $k$-choosable} or \emph{$k$-list-colorable}. The \emph{list chromatic number} of a graph $G$ is the smallest integer $k$ such that $G$ is $k$-choosable. List coloring can be very different from usual coloring as there exist graphs with a small chromatic number and an arbitrarily large list chromatic number.

\cite{kramer2,kramer1} introduced the notion of 2-distance coloring. This notion generalizes the ``proper'' constraint (that does not allow two adjacent vertices to have the same color) in the following way: a \emph{$2$-distance $k$-coloring} is such that no pair of vertices at distance at most 2 have the same color (similarly to proper $k$-list-coloring, one can also define \emph{$2$-distance $k$-list-coloring}). The \emph{$2$-distance chromatic number} of $G$, denoted by $\chi^2(G)$, is the smallest integer $k$ so that $G$ has a 2-distance $k$-coloring.

For all $v\in V$, we denote $d_G(v)$ the degree of $v$ in $G$ and by
$\Delta(G) = \max_{v\in V}d_G(v)$ the maximum degree of a graph
$G$. For brevity, when it is clear from the context, we will use
$\Delta$ (resp. $d(v)$) instead of $\Delta(G)$ (resp. $d_G(v)$).  One
can observe that, for any graph $G$, $\Delta+1\leq\chi^2(G)\leq
\Delta^2+1$. The lower bound is trivial since, in a 2-distance
coloring, every neighbor of a vertex $v$ with degree $\Delta$, and $v$
itself must have a different color. As for the upper bound, a greedy
algorithm shows that $\chi^2(G)\leq \Delta^2+1$. Moreover, this bound
is tight for some graphs, for example, Moore graphs of type
$(\Delta,2)$, which are graphs where all vertices have degree
$\Delta$, are at distance at most two from each other, and the total
number of vertices is $\Delta^2+1$. See \Cref{tight upper bound
  figure}. Also,  incidence graphs of finite projective planes give
the inequality $\chi^2(G)\ge \Delta^2-\Delta+1$ when $\Delta - 1$ is a prime power \citep{brown1966graphs}.

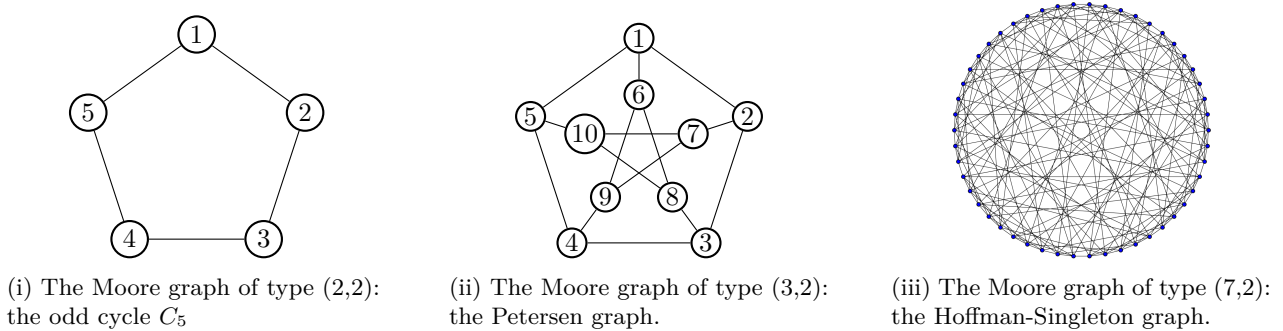
\begin{figure}[!htbp]
\begin{center}
\subfigure[The Moore graph of type (2,2): the odd cycle $C_5$.]{
\begin{tikzpicture}[every node/.style={circle,thick,draw,minimum size=1pt,inner sep=2}]
  \graph[clockwise, radius=1.5cm] {subgraph C_n [n=5,name=A] };
\end{tikzpicture}
}
\hfil
\subfigure[The Moore graph of type (3,2): the Petersen graph.]{
\begin{tikzpicture}[every node/.style={circle,thick,draw,minimum size=1pt,inner sep=1}]
  \graph[clockwise, radius=1.5cm] {subgraph C_n [n=5,name=A] };
  \graph[clockwise, radius=0.75cm,n=5,name=B] {1/"6", 2/"7", 3/"8", 4/"9", 5/"10" };

  \foreach \i [evaluate={\j=int(mod(\i+6,5)+1)}]
     in {1,2,3,4,5}{
    \draw (A \i) -- (B \i);
    \draw (B \j) -- (B \i);
  }
\end{tikzpicture}
}
\hfil
\subfigure[The Moore graph of type (7,2): the Hoffman-Singleton graph.]{
\includegraphics[scale=0.12]{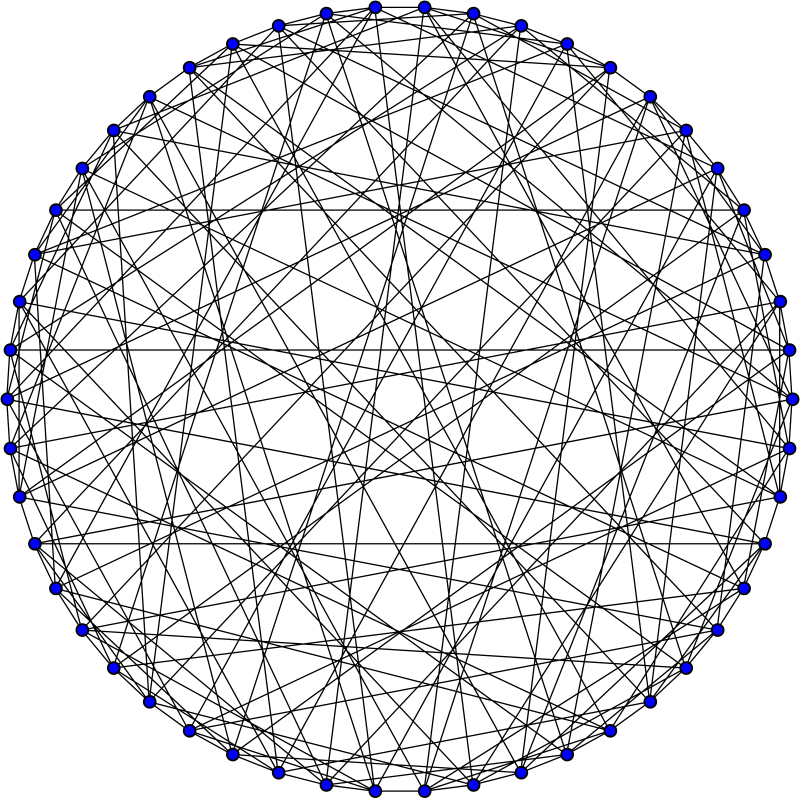}
}
\caption{Examples of Moore graphs for which $\chi^2=\Delta^2+1$.}
\label{tight upper bound figure}
\end{center}
\end{figure}

By nature, $2$-distance colorings and the $2$-distance chromatic number of a graph depend a lot on the number of vertices in the neighborhood of every vertex. More precisely, the ``sparser'' a graph is, the lower its $2$-distance chromatic number will be. One way to quantify the sparsity of a graph is through its maximum average degree. The \emph{average degree} $\ad$ of a graph $G=(V,E)$ is defined by $\ad(G)=\frac{2|E|}{|V|}$. The \emph{maximum average degree} $\mad(G)$ is the maximum, over all subgraphs $H$ of $G$, of $\ad(H)$. Another way to measure the sparsity when the graph is planar (a graph is \emph{planar} if one can draw its vertices with points on the plane, and edges with curves intersecting only at its endpoints) is through the girth, i.e. the length of a shortest cycle. We denote $g(G)$ the girth of $G$. Intuitively, the higher the girth of a planar graph is, the sparser it gets. 


When $G$ is a planar graph, Wegner conjectured in 1977 that  $\chi^2(G)$ becomes linear in $\Delta(G)$:

\begin{conjecture}[\cite{wegner}]
\label{conj:Wegner}
Let $G$ be a planar graph with maximum degree $\Delta$. Then,
$$
\chi^2(G) \leq \left\{
    \begin{array}{ll}
        7, & \mbox{if } \Delta\leq 3, \\
        \Delta + 5, & \mbox{if } 4\leq \Delta\leq 7,\\
        \left\lfloor\frac{3\Delta}{2}\right\rfloor + 1, & \mbox{if } \Delta\geq 8.
    \end{array}
\right.
$$
\end{conjecture}

The upper bound for the case where $\Delta\geq 8$ is tight (see \Cref{wegner figure}(i)). Recently, the case $\Delta\leq 3$ was proved by Thomassen \citep{tho18}, and by Hartke \textit{et al.} \citep{har16} independently. For $\Delta\geq 8$, Havet \textit{et al.} \citep{havet} proved that the bound is $\frac{3}{2}\Delta(1+o(1))$, where $o(1)$ is as $\Delta\rightarrow\infty$ (this bound holds for 2-distance list-colorings). \Cref{conj:Wegner} is known to be true for some subfamilies of planar graphs, for example $K_4$-minor free graphs \citep{lwz03}.

\begin{figure}[htbp]
\begin{center}
\subfigure[A graph with girth 3 and $\chi^2=\lfloor\frac{3\Delta}{2}\rfloor+1$]{
\begin{tikzpicture}[scale=0.4]
\begin{scope}[every node/.style={circle,thick,draw,minimum size=1pt,inner sep=2}]
    \node[fill] (y) at (0,0) {};
    \node[fill] (z) at (5,0) {};
    \node[fill] (x) at (2.5,4.33) {};

    \node[fill] (xy) at (1.25,2.165) {};
    \node[fill] (yz) at (2.5,0) {};
    \node[fill] (zx) at (3.75,2.165) {};

    \node[fill,label={[label distance=-1cm]above:$\lfloor\frac{\Delta}{2}\rfloor-1$ vertices}] (xy1) at (-2.5,4.33) {};
    \node[fill] (xy2) at (-1.5625,3.78875) {};
    \node[fill] (xy3) at (-0.625,3.2475) {};

    \node[fill,label={[label distance=-0.7cm]above:$\lceil\frac{\Delta}{2}\rceil$ vertices}] (zx1) at (7.5,4.33) {};
    \node[fill] (zx2) at (6.5625,3.78875) {};
    \node[fill] (zx3) at (5.625,3.2475) {};

    \node[fill,label={[label distance = -0.7cm]below:$\lfloor\frac{\Delta}{2}\rfloor$ vertices}] (yz1) at (2.5,-4.33) {};
    \node[fill] (yz2) at (2.5,-3.2475) {};
    \node[fill] (yz3) at (2.5,-2.165) {};

\end{scope}

\begin{scope}[every edge/.style={draw=black,thick}]
    \path (x) edge (y);
    \path (y) edge (z);
    \path (z) edge (x);

    \path (x) edge[bend left] (y);

    \path (x) edge (xy1);
    \path (x) edge (xy2);
    \path (x) edge (xy3);

    \path (y) edge (xy1);
    \path (y) edge (xy2);
    \path (y) edge (xy3);

    \path (x) edge (zx1);
    \path (x) edge (zx2);
    \path (x) edge (zx3);

    \path (z) edge (zx1);
    \path (z) edge (zx2);
    \path (z) edge (zx3);

    \path (y) edge (yz1);
    \path (y) edge (yz2);
    \path (y) edge (yz3);

    \path (z) edge (yz1);
    \path (z) edge (yz2);
    \path (z) edge (yz3);

    \path[dashed] (xy) edge (xy3);
    \path[dashed] (yz) edge (yz3);
    \path[dashed] (zx) edge (zx3);
\end{scope}
\draw[rotate=-30] (-0.625-1.4,3.2475-0.7) ellipse (3cm and 0.5cm);
\draw[rotate=30] (5+0.625+1,3.2475-3.25) ellipse (3cm and 0.5cm);
\draw[rotate=90] (-2,-2.5) ellipse (3cm and 0.5cm);
\end{tikzpicture}
}
\hfil
\subfigure[A graph with girth 4 and $\chi^2=\lfloor\frac{3\Delta}{2}\rfloor-1$.]{
\begin{tikzpicture}[scale=0.4]
\begin{scope}[every node/.style={circle,thick,draw,minimum size=1pt,inner sep=2}]
    \node[fill] (y) at (0,0) {};
    \node[fill] (z) at (5,0) {};
    \node[fill] (x) at (2.5,4.33) {};

    \node[fill] (xy) at (1.25,2.165) {};
    \node[fill] (yz) at (2.5,0) {};
    \node[fill] (zx) at (3.75,2.165) {};

    \node[fill,label={[label distance=-1cm]above:$\lfloor\frac{\Delta}{2}\rfloor-1$ vertices}] (xy1) at (-2.5,4.33) {};
    \node[fill] (xy2) at (-1.5625,3.78875) {};
    \node[fill] (xy3) at (-0.625,3.2475) {};

    \node[fill,label={[label distance=-0.7cm]above:$\lceil\frac{\Delta}{2}\rceil$ vertices}] (zx1) at (7.5,4.33) {};
    \node[fill] (zx2) at (6.5625,3.78875) {};
    \node[fill] (zx3) at (5.625,3.2475) {};

    \node[fill,label={[label distance = -0.7cm]below:$\lfloor\frac{\Delta}{2}\rfloor$ vertices}] (yz1) at (2.5,-4.33) {};
    \node[fill] (yz2) at (2.5,-3.2475) {};
    \node[fill] (yz3) at (2.5,-2.165) {};

\end{scope}

\begin{scope}[every edge/.style={draw=black,thick}]
    \path (x) edge (y);
    \path (y) edge (z);
    \path (z) edge (x);

    \path (x) edge (xy1);
    \path (x) edge (xy2);
    \path (x) edge (xy3);

    \path (y) edge (xy1);
    \path (y) edge (xy2);
    \path (y) edge (xy3);

    \path (x) edge (zx1);
    \path (x) edge (zx2);
    \path (x) edge (zx3);

    \path (z) edge (zx1);
    \path (z) edge (zx2);
    \path (z) edge (zx3);

    \path (y) edge (yz1);
    \path (y) edge (yz2);
    \path (y) edge (yz3);

    \path (z) edge (yz1);
    \path (z) edge (yz2);
    \path (z) edge (yz3);

    \path[dashed] (xy) edge (xy3);
    \path[dashed] (yz) edge (yz3);
    \path[dashed] (zx) edge (zx3);
\end{scope}
\draw[rotate=-30] (-0.625-1.4,3.2475-0.7) ellipse (3cm and 0.5cm);
\draw[rotate=30] (5+0.625+1,3.2475-3.25) ellipse (3cm and 0.5cm);
\draw[rotate=90] (-2,-2.5) ellipse (3cm and 0.5cm);
\end{tikzpicture}
}
\caption{Graphs with $\chi^2\approx \frac32 \Delta$}
\label{wegner figure}
\end{center}
\end{figure}
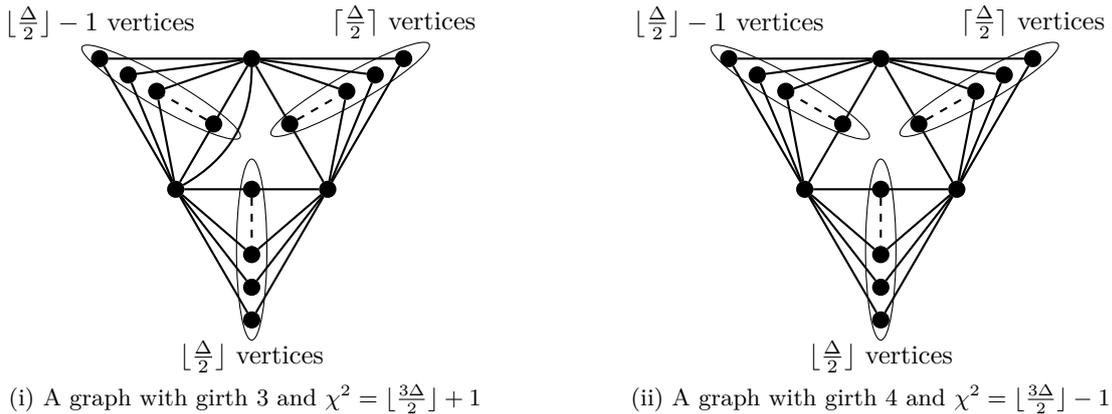

Wegner's conjecture motivated extensive researches on $2$-distance chromatic number of sparse graphs, either of planar graphs with high girth or of graphs with upper bounded maximum average degree which are directly linked due to \Cref{maximum average degree and girth proposition}.

\begin{proposition}[Folklore]\label{maximum average degree and girth proposition}
For every planar graph $G$, $(\mad(G)-2)(g(G)-2)<4$.
\end{proposition}

As a consequence, any theorem with an upper bound on $\mad(G)$ can be translated to a theorem with a lower bound on $g(G)$ under the condition that $G$ is planar.

Many results have taken the following form: \textit{every planar graph $G$ of girth $g(G)\geq g_0$ and $\Delta(G)\geq \Delta_0$ satisfies $\chi^2(G)\leq \Delta(G)+c(g_0,\Delta_0)$ where $c(g_0,\Delta_0)$ is a small constant depending only on $g_0$ and $\Delta_0$}. Due to \Cref{maximum average degree and girth proposition}, these type of results sometimes come as a corollary of the same result on graphs with bounded maximum average degree. \Cref{recap table 2-distance} shows all known such results, up to our knowledge, on the $2$-distance chromatic number of planar graphs with fixed girth, either proven directly for planar graphs with high girth or came as a corollary of a result on graphs with bounded maximum average degree.

\begin{table}[htbp]
\begin{center}
\scalebox{0.685}{%
\begin{tabular}{||c||c|c|c|c|c|c|c|c||}
\hline
\backslashbox{$g_0$ \kern-1em}{\kern-1em $\chi^2(G)$} & $\Delta+1$ & $\Delta+2$ & $\Delta+3$ & $\Delta+4$ & $\Delta+5$ & $\Delta+6$ & $\Delta+7$ & $\Delta+8$\\
\hline \hline
$3$ & \slashbox{\phantom{\ \ \ \ \ }}{} & & &$\Delta=3$ \tablefootnote{\citep{tho18,har16}}& & & & \\
\hline
$4$ & \slashbox{\phantom{\ \ \ \ \ }}{} & & & & & & & \\
\hline
$5$ & \slashbox{\phantom{\ \ \ \ \ }}{} &$\Delta\geq 10^7$ \tablefootnote{\citep{bon19}}\footref{list footnote} &$\Delta\geq 339$ \tablefootnote{\citep{don17b}} &$\Delta\geq 312$ \tablefootnote{\citep{don17}} &$\Delta\geq 15$ \tablefootnote{\citep{bu18b}}\tablefootnote{\label{other footnote}Corollaries of more general colorings of planar graphs.} &$\Delta\geq 12$ \tablefootnote{\citep{bu16}}\footref{list footnote} & $\Delta\neq 7,8$ \tablefootnote{\citep{don17}} &all $\Delta$ \tablefootnote{\citep{dl16}}\\
\hline
$6$ & \slashbox{\phantom{\ \ \ \ \ }}{} &$\Delta\geq 17$ \tablefootnote{\citep{bon14}}\footref{mad footnote} &$\Delta\geq 9$ \tablefootnote{\citep{bu16}}\footref{list footnote} & &all $\Delta$ \tablefootnote{\citep{bu11}} & & & \\
\hline
$7$ & $\Delta\geq 16$ \tablefootnote{\citep{iva11}}\tablefootnote{\label{list footnote}Corollaries of 2-distance list-colorings of planar graphs.} & & &$\Delta=4$ \tablefootnote{\citep{cra13}}\tablefootnote{\label{mad list footnote}Corollaries of 2-distance list-colorings of graphs with a bounded maximum average degree.} & & & & \\
\hline
$8$ & $\Delta\geq 9$ \tablefootnote{\citep{lmpv19}}\footref{other footnote}& &$\Delta=5$ \tablefootnote{\citep{bu15}}\footref{mad list footnote} & & & & & \\
\hline
$9$ &$\Delta\geq 7$ \tablefootnote{\citep{lm21}}\footref{mad footnote} &$\Delta=5$ \tablefootnote{\citep{bu15}}\footref{mad list footnote} &$\Delta=3$ \tablefootnote{\citep{cra07}}\footref{list footnote} & & & & & \\
\hline
$10$ & $\Delta\geq 6$ \tablefootnote{\citep{iva11}}\footref{list footnote} & & & & & & & \\
\hline
 $11$ & &$\Delta=4$ \tablefootnote{\citep{cra13}}\footref{mad list footnote} & & & & & & \\
\hline
$12$ & $\Delta=5$ \tablefootnote{\citep{iva11}}\footref{list footnote} &$\Delta=3$ \tablefootnote{\citep{bi12}}\footref{list footnote} & & & & & &\\
\hline
$13$ & & & & & & & &\\
\hline
$14$ &$\Delta\geq 4$ \tablefootnote{\citep{bon13}}\tablefootnote{\label{mad footnote}Corollaries of 2-distance colorings of graphs with a bounded maximum average degree.} & & & & & & &\\
\hline
$\dots$ & & & & & & & & \\
\hline
$21$ & \textcolor{red}{$\Delta = 3$} \tablefootnote{Our result.} & & & & & & & \\
\hline
$22$ & $\Delta=3$ \tablefootnote{\citep{bi12bis}} & & & & & & & \\
\hline
\end{tabular}}
\caption{The latest results with a coefficient 1 before $\Delta$ in the upper bound of $\chi^2$.}
\label{recap table 2-distance}
\end{center}
\end{table} 

For example, the result from line ``7'' and column ``$\Delta + 1$'' from \Cref{recap table 2-distance} reads as follows : ``\emph{every planar graph $G$ of girth at least 7  and of $\Delta$ at least 16 satisfies $\chi^2(G)\leq \Delta+1$}''. The crossed out cases in the first column correspond to the fact that, for $g_0\leq 6$, there are planar graphs $G$ with $\chi^2(G)=\Delta+2$ for arbitrarily large $\Delta$ \citep{bor04,dvo08b}. The lack of results for $g = 4$ is due to the fact that the graph in \Cref{wegner figure}(ii) has girth 4, and $\chi^2=\lfloor\frac{3\Delta}{2}\rfloor-1$ for all $\Delta$.

We are interested in the case $\chi^2(G)=\Delta+1$ as $\Delta+1$ is a trivial lower bound for $\chi^2(G)$. In particular, we are interested in planar \emph{subcubic} graphs, which are graphs with maximum degree $\Delta=3$. More precisely, we are trying to answer the following question:

\begin{question} \label{question}
What is the smallest $g_0$ such that every planar subcubic graph $G$ with girth $g(G)\geq g_0$ verifies $\chi^2(G)\leq 4$?
\end{question}

This question was first looked at in \citep{bor04} by Borodin \textit{et al.} where the authors proved that $g_0\leq 24$. Later on, Borodin and Ivanova improved the upper bound on $g_0$ to 23 in \citep{bor11}, then 22 in \citep{bi12bis}. In this article, we aim to prove that $g_0$ is at most 21. All of these results rely on the fact that there are only 4 colors in total, an approach that cannot be generalized to list coloring.

\begin{theorem} \label{main theorem}
If $G$ is a planar subcubic graph with $g(G)\geq 21$, then $\chi^2(G)\leq 4$.
\end{theorem}

In \Cref{sec2}, we present the proof of \Cref{main theorem} using the well-known discharging method. The reducible configurations are obtained by further exploiting the techniques presented in \citep{bi12bis}.

There was also another approach to \Cref{question}, that is to find lower bounds on $g_0$. While construction of planar graphs with $\chi^2(G)\geq \Delta+2$ for any $\Delta$ is known for small girth \citep{bor04,dvo08b}. The first construction with high girth ($g_0\geq 9$) was presented by Dvo\u{r}ak \textit{et al.} in \citep{dvo08} where the authors relied on an interesting property of $2$-distance $4$-colorings of vertices at distance 5 from each other. In \Cref{sec3}, we improve further upon this idea to build a planar subcubic graph of girth 11 with $\chi^2(G)\geq 5$. In other words, we improved the lower bound on $g_0$ from 9 to 11. 

\section{Proof of \texorpdfstring{\Cref{main theorem}}{Theorem 4}}
\label{sec2}

\paragraph{Notations and drawing conventions.} For $v\in V(G)$, the \emph{2-distance neighborhood} of $v$, denoted $N^*_G(v)$, is the set of 2-distance neighbors of $v$, which are vertices at distance at most two from $v$, not including $v$. We also denote $d^*_G(v)=|N^*_G(v)|$. We call $F(G)$ the set of faces of $G$ and for all $f\in F(G)$, $d_G(f)$ is the size of face $f$ (bridges are counted twice). We will drop the subscript and the argument when it is clear from the context. Also for conciseness, from now on, when we say ``to color'' a vertex, it means to color such vertex differently from all of its colored neighbors at distance at most two. Similarly, any considered coloring will be a 2-distance coloring. We will also say that a vertex $u$ ``sees'' another vertex $v$ if $u$ and $v$ are at distance at most 2 from each other. 

Some more notations:
\begin{itemize}
\item A \emph{$d$-vertex} is a vertex of degree $d$.
\item A \emph{$k$-path} ($k^+$-path, $k^-$-path) is a path of length $k+1$ (at least $k+1$, at most $k+1$) where the $k$ internal vertices are 2-vertices and the endvertices are $3$-vertices.
\item We denote \emph{$(k,l,m)$} a $3$-vertex incident to a $k$-path, an $l$-path, and an $m$-path.
\item A pair of vertices $(k^+,l^+,m)$ and $(m,n^+,p^+)$ joined by an $m$-path will be denoted by $(klm-mnp)$. Similarly, a triple of vertices $u=(k^+,l^+,m)$, $v=(m,n^+,p)$, and $w=(p,q^+,r^+)$ where $u$ and $v$ are joined by an $m$-path and $v$ and $w$ are joined by a $p$-path, will be denoted by $(klm-mnp-pqr)$. This notation is taken from \citep{bi12bis}.
\end{itemize}
As a drawing convention for the rest of the figures, black vertices will have a fixed degree, which is represented, and white vertices may have a higher degree than what is drawn.

Let $G$ be a counterexample to \Cref{main theorem} minimizing $|V(G)|+|E(G)|$. Recall that every cycle except $C_5$ is colorable with 4 colors hence, since $G$ has girth at least 21, it has maximum degree $\Delta=3$. The purpose of the proof is to prove that $G$ cannot exist.
The main idea of the proof relies on studying configurations (graphs with given list sizes) that are ``almost'' but not colorable. Using the fact that we have only 4 colors in total, we are able to deduce the exact content of the lists, thus allowing us to reduce new configurations and improve upon the previous results.
In the following sections, we will study the structural properties of $G$ (\Cref{tutu}). We will then apply a discharging procedure (\Cref{tonton}). The discharging argument captures the sparseness of the graph, meaning that one of our reducible configurations must appear. More formally, we have due to the Euler formula ($|V|-|E|+|F|=2$):
\begin{equation}\label{equation}
\sum_{u\in V(G)} \left(\frac{19}2d(u)-21\right) + \sum_{f\in F(G)} \left(d(f)-21\right) = -42 < 0
\end{equation}

We assign to each vertex $u$ the charge $\mu(u)=\frac{19}2d(u)-21$ and to each face $f$ the charge $\mu(f)=d(f)-21$. To prove the non-existence of $G$, we will redistribute the charges preserving their sum and obtaining a non-negative total charge, which will contradict \Cref{equation}.

\subsection{Useful observations}

Before studying the structural properties of $G$, we will introduce some useful observations and lemmas that will be the core of the reducibility proofs of our configurations.

For a vertex $u$, let $L(u)$ denote the set of available colors for $u$ from the set $\{a,b,c,d\}$. For convenience, the lower bound on $|L(u)|$ will be depicted on the figures below the corresponding vertex $u$. 

\begin{lemma} \label{122 lemma}
The graphs depicted in \Cref{122 figure} are colorable unless their lists of colors are exactly what is indicated.
\end{lemma}

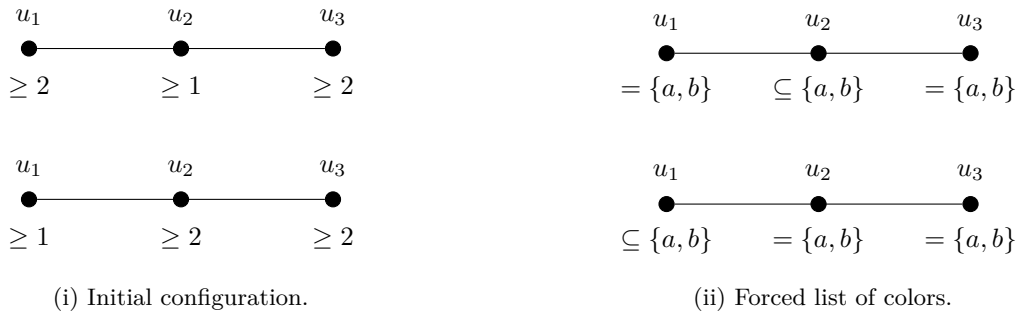
\begin{figure}[htbp]
\begin{center}
\subfigure[Initial configuration.]{
\begin{tikzpicture}{thick}
\begin{scope}[every node/.style={circle,draw,minimum size=1pt,inner sep=2}]
    \node[fill,label={above:$u_1$},label={below:$\geq 1$}] (1) at (0,0) {};
    \node[fill,label={above:$u_2$},label={below:$\geq 2$}] (2) at (2,0) {};
    \node[fill,label={above:$u_3$},label={below:$\geq 2$}] (3) at (4,0) {};
    
    \node[fill,label={above:$u_1$},label={below:$\geq 2$}] (10) at (0,2) {};
    \node[fill,label={above:$u_2$},label={below:$\geq 1$}] (20) at (2,2) {};
    \node[fill,label={above:$u_3$},label={below:$\geq 2$}] (30) at (4,2) {};
\end{scope}

\begin{scope}[every edge/.style={draw=black}]
    \path (1) edge (3);
    \path (10) edge (30);
\end{scope}
\end{tikzpicture}
}
\hfil
\subfigure[Forced list of colors.]{
\begin{tikzpicture}{thick}
\begin{scope}[every node/.style={circle,draw,minimum size=1pt,inner sep=2}]
    \node[fill,label={above:$u_1$},label={[label distance = -10pt]below:$\subseteq\{a,b\}$}] (1) at (0,0) {};
    \node[fill,label={above:$u_2$},label={[label distance = -10pt]below:$=\{a,b\}$}] (2) at (2,0) {};
    \node[fill,label={above:$u_3$},label={[label distance = -10pt]below:$=\{a,b\}$}] (3) at (4,0) {};
    
    \node[fill,label={above:$u_1$},label={[label distance = -10pt]below:$=\{a,b\}$}] (10) at (0,2) {};
    \node[fill,label={above:$u_2$},label={[label distance = -10pt]below:$\subseteq \{a,b\}$}] (20) at (2,2) {};
    \node[fill,label={above:$u_3$},label={[label distance = -10pt]below:$=\{a,b\}$}] (30) at (4,2) {};
\end{scope}

\begin{scope}[every edge/.style={draw=black}]
    \path (1) edge (3);
    \path (10) edge (30);
\end{scope}
\end{tikzpicture}
}
\caption{An useful non-colorable graph on three vertices.}
\label{122 figure}
\end{center}
\end{figure}

\begin{proof}
If $|L(u_1)\cup L(u_2)\cup L(u_3)|\geq 3$, then $u_1$, $u_2$, and
$u_3$ are easily colorable (by Hall's theorem by example). Thus, we
can assume without loss of generality (w.l.o.g. for short) that $L(u_i)\subseteq
\{a,b\}$ for all $1\leq i\leq 3$.
\end{proof}

\begin{lemma} \label{restriction lemma}
Let $H$ be a graph on $n\geq 4$ vertices $u_1,u_2,\dots,u_n$. Let the degree and adjacency of $u_1$, $u_2$, and $u_3$ be as depicted in \Cref{restriction figure}. Let $|L(u_1)|\geq 2$, $|L(u_2)|\geq 3$, and $|L(u_3)|\geq d^*_H(u_3)-1$. If, for every $x$ in $L(u_4)$, we have that $u_4,u_5,\dots,u_n$ are colorable with the respective lists $L(u_4)\setminus\{x\},L(u_5),L(u_6),\dots,L(u_n)$, then $H$ is colorable.
\end{lemma}

\begin{figure}[htbp]
\begin{center}
\begin{tikzpicture}{thick}
\begin{scope}[every node/.style={circle,draw,minimum size=1pt,inner sep=2}]
    \node[fill,label={above:$u_1$},label={below:$\geq 2$}] (1) at (0,0) {};
    \node[fill,label={above:$u_2$},label={below:$\geq 3$}] (2) at (2,0) {};
    \node[fill,label={above:$u_3$},label={[label distance = -10pt]below:$\geq d^*-1$}] (3) at (4,0) {};
    \node[label={above:$u_4$}] (4) at (6,0) {};
\end{scope}

\begin{scope}[every edge/.style={draw=black}]
    \path (1) edge (4);
\end{scope}
\end{tikzpicture}
\caption{Graph $H$ from \Cref{restriction lemma}.}
\label{restriction figure}
\end{center}
\end{figure}

\begin{proof}
Suppose by contradiction that $H$ is not colorable. We remove the extra colors from $L(u_1)$ and $L(u_2)$ so that $|L(u_1)|=2$ and $|L(u_2)|=3$. We choose $x\in L(u_2)\setminus L(u_1)$. By hypothesis, there exists a coloring of $u_4,u_5,\dots,u_n$ where $u_4$ is not colored $x$. The remaining vertices, namely $u_1$, $u_2$, and $u_3$ must not be colorable. Since $|L(u_1)|\geq 2$, $|L(u_2)|\geq 3$, and $|L(u_3)|\geq d^*_H(u_3)-1$, after coloring $u_4,\dots,u_n$, the lists of available colors for $u_1$, $u_2$, and $u_3$ verify $|L(u_1)|\geq 2$, $|L(u_2)|\geq 2$, and $|L(u_3)|\geq 1$. Since they are not colorable, by \Cref{122 lemma}, $L(u_1)=L(u_2)$. However, this is impossible since $x\in L(u_2)\setminus L(u_1)$ initially and $x$ remains in $L(u_2)$ since $u_4$ was not colored $x$. 
\end{proof}

\begin{observation}
\Cref{restriction lemma} means that, by restricting the list $L(u_4)$ to $L(u_4)\setminus\{x\}$ for a well chosen color $x\in L(u_4)$, we can always color $u_1$, $u_2$, and $u_3$ last. As a result, if $H-\{u_1,u_2,u_3\}$ is colorable with $L'(u_4)$ where $|L'(u_4)|=|L(u_4)|-1$ and $L'(u_4)\subset L(u_4)$ ($L'(u_i)=L(u_i)$ for all $5\leq i\leq n$), then $H$ is colorable. From now on, for convenience, we will say that we restrict $u_4$ by one color to color $u_1$, $u_2$, and $u_3$ afterwards.
\end{observation}

\begin{lemma} \label{colorable lemma}
The graphs depicted in \Cref{colorable figure} are all colorable.
\end{lemma}

\begin{figure}[htbp]
\begin{center}
\subfigure[\label{2232}]{
\begin{tikzpicture}[scale=0.55]{thick}
\begin{scope}[every node/.style={circle,draw,minimum size=1pt,inner sep=2}]
    \node[fill,label={above:$u_1$},label={below:$\geq 2$}] (1) at (0,0) {};
    \node[fill,label={above:$u_2$},label={below:$\geq 2$}] (2) at (2,0) {};
    \node[fill,label={above:$u_3$},label={below:$\geq 3$}] (3) at (4,0) {};
    \node[fill,label={above:$u_4$},label={below:$\geq 2$}] (4) at (6,0) {};
\end{scope}

\begin{scope}[every edge/.style={draw=black}]
    \path (1) edge (4);
\end{scope}
\end{tikzpicture}
}
\hfil
\subfigure[\label{2343}]{
\begin{tikzpicture}[scale=0.55]{thick}
\begin{scope}[every node/.style={circle,draw,minimum size=1pt,inner sep=2}]
    \node[fill,label={above:$u_1$},label={below:$\geq 2$}] (1) at (0,0) {};
    \node[fill,label={[label distance = +4pt]above right:$u_2$},label={below:$\geq 3$}] (2) at (2,0) {};
    \node[fill,label={[label distance = +4pt]above right:$u_3$},label={below:$\geq 4$}] (3) at (4,0) {};
    \node[fill,label={above:$u_4$},label={below:$\geq 3$}] (4) at (6,0) {};
    
    \node[fill,label={above:$u'_2$},label={below left:$\geq 3$}] (2') at (2,2) {};
    \node[fill,label={above:$u'_3$},label={below left:$\geq 3$}] (3') at (4,2) {};
\end{scope}

\begin{scope}[every edge/.style={draw=black}]
    \path (1) edge (4);
    \path (2) edge (2');
    \path (3) edge (3');
\end{scope}
\end{tikzpicture}
}
\hfil
\subfigure[\label{22332}]{
\begin{tikzpicture}[scale=0.55]{thick}
\begin{scope}[every node/.style={circle,draw,minimum size=1pt,inner sep=2}]
    \node[fill,label={above:$u_1$},label={below:$\geq 2$}] (1) at (0,0) {};
    \node[fill,label={above:$u_2$},label={below:$\geq 2$}] (2) at (2,0) {};
    \node[fill,label={above:$u_3$},label={below:$\geq 3$}] (3) at (4,0) {};
    \node[fill,label={above:$u_4$},label={below:$\geq 3$}] (4) at (6,0) {};
    \node[fill,label={above:$u_5$},label={below:$\geq 2$}] (5) at (8,0) {};
\end{scope}

\begin{scope}[every edge/.style={draw=black}]
    \path (1) edge (5);
\end{scope}
\end{tikzpicture}
}
\hfil
\subfigure[\label{23432} $L(u_1)\subseteq L(u_2)$ and $L(u_5)\subseteq L(u_4)$.]{
\begin{tikzpicture}[scale=0.6]{thick}
\begin{scope}[every node/.style={circle,draw,minimum size=1pt,inner sep=2}]
    \node[fill,label={above:$u_1$},label={below:$\geq 2$}] (1) at (0,0) {};
    \node[fill,label={above:$u_2$},label={below:$\geq 3$}] (2) at (2,0) {};
    \node[fill,label={[label distance = +4pt]above right:$u_3$},label={below:$\geq 4$}] (3) at (4,0) {};
    \node[fill,label={above:$u_4$},label={below:$\geq 3$}] (4) at (6,0) {};
    \node[fill,label={above:$u_5$},label={below:$\geq 2$}] (5) at (8,0) {};
    
    \node[fill,label={above:$u'_3$},label={below left:$\geq 2$}] (3') at (4,2) {};
\end{scope}

\begin{scope}[every edge/.style={draw=black}]
    \path (1) edge (5);
    \path (3) edge (3');
\end{scope}
\end{tikzpicture}
}
\hfil
\subfigure[\label{23342}]{
\begin{tikzpicture}[scale=0.7]{thick}
\begin{scope}[every node/.style={circle,draw,minimum size=1pt,inner sep=2}]
    \node[fill,label={above:$u_1$},label={below:$\geq 2$}] (1) at (0,0) {};
    \node[fill,label={above:$u_2$},label={below:$\geq 3$}] (2) at (2,0) {};
    \node[fill,label={[label distance = +4pt]above right:$u_3$},label={below:$\geq 3$}] (3) at (4,0) {};
    \node[fill,label={above:$u_4$},label={below:$\geq 4$}] (4) at (6,0) {};
    \node[fill,label={above:$u_5$},label={below:$\geq 2$}] (5) at (8,0) {};
    
    \node[fill,label={above:$u'_3$},label={below left:$\geq 2$}] (3') at (4,2) {};
\end{scope}

\begin{scope}[every edge/.style={draw=black}]
    \path (1) edge (5);
    \path (3) edge (3');
\end{scope}
\end{tikzpicture}
}
\hfil
\subfigure[\label{22432}]{
\begin{tikzpicture}[scale=0.7]{thick}
\begin{scope}[every node/.style={circle,draw,minimum size=1pt,inner sep=2}]
    \node[fill,label={above:$u_1$},label={below:$\geq 2$}] (1) at (0,0) {};
    \node[fill,label={above:$u_2$},label={below:$\geq 2$}] (2) at (2,0) {};
    \node[fill,label={[label distance = +4pt]above right:$u_3$},label={below:$\geq 4$}] (3) at (4,0) {};
    \node[fill,label={above:$u_4$},label={below:$\geq 3$}] (4) at (6,0) {};
    \node[fill,label={above:$u_5$},label={below:$\geq 2$}] (5) at (8,0) {};
    
    \node[fill,label={above:$u'_3$},label={below left:$\geq 3$}] (3') at (4,2) {};
\end{scope}

\begin{scope}[every edge/.style={draw=black}]
    \path (1) edge (5);
    \path (3) edge (3');
\end{scope}
\end{tikzpicture}
}
\hfil
\subfigure[\label{24343}$L(u'_3)\subseteq L(u_3)$.]{
\begin{tikzpicture}[scale=0.7]{thick}
\begin{scope}[every node/.style={circle,draw,minimum size=1pt,inner sep=2}]
    \node[fill,label={above:$u_1$},label={below:$\geq 2$}] (1) at (0,0) {};
    \node[fill,label={above:$u_2$},label={below:$\geq 4$}] (2) at (2,0) {};
    \node[fill,label={[label distance = +4pt]above right:$u_3$},label={below:$\geq 3$}] (3) at (4,0) {};
    \node[fill,label={[label distance = +4pt]above right:$u_4$},label={below:$\geq 4$}] (4) at (6,0) {};
    \node[fill,label={above:$u_5$},label={below:$\geq 3$}] (5) at (8,0) {};
    
    \node[fill,label={above:$u'_3$},label={below left:$\geq 2$}] (3') at (4,2) {};
    \node[fill,label={above:$u'_4$},label={below left:$\geq 3$}] (4') at (6,2) {};
\end{scope}

\begin{scope}[every edge/.style={draw=black}]
    \path (1) edge (5);
    \path (3) edge (3');
    \path (4) edge (4');
\end{scope}
\end{tikzpicture}
}
\hfil
\subfigure[\label{23432bis}$L(u_1)\subseteq L(u_2)$, $L(u''_3)\subseteq L(u'_3)$, and $L(u_5)\subseteq L(u_4)$.]{
\begin{tikzpicture}[scale=0.7]{thick}
\begin{scope}[every node/.style={circle,draw,minimum size=1pt,inner sep=2}]
    \node[fill,label={above:$u_1$},label={below:$\geq 2$}] (1) at (0,0) {};
    \node[fill,label={above:$u_2$},label={below:$\geq 3$}] (2) at (2,0) {};
    \node[fill,label={[label distance = +4pt]above right:$u_3$},label={below:$\geq 4$}] (3) at (4,0) {};
    \node[fill,label={above:$u_4$},label={below:$\geq 3$}] (4) at (6,0) {};
    \node[fill,label={above:$u_5$},label={below:$\geq 2$}] (5) at (8,0) {};
    
    \node[fill,label={[label distance = +4pt]above right:$u'_3$},label={below left:$\geq 3$}] (3') at (4,2) {};
    \node[fill,label={above:$u''_3$},label={below left:$\geq 2$}] (3'') at (4,4) {};
\end{scope}

\begin{scope}[every edge/.style={draw=black}]
    \path (1) edge (5);
    \path (3) edge (3'');
\end{scope}
\end{tikzpicture}
}
\hfil
\subfigure[\label{224322}]{
\begin{tikzpicture}[scale=0.7]{thick}
\begin{scope}[every node/.style={circle,draw,minimum size=1pt,inner sep=2}]
    \node[fill,label={above:$u_1$},label={below:$\geq 2$}] (1) at (0,0) {};
    \node[fill,label={above:$u_2$},label={below:$\geq 2$}] (2) at (2,0) {};
    \node[fill,label={above:$u_3$},label={below:$\geq 4$}] (3) at (4,0) {};
    \node[fill,label={above:$u_4$},label={below:$\geq 3$}] (4) at (6,0) {};
    \node[fill,label={above:$u_5$},label={below:$\geq 2$}] (5) at (8,0) {};
    \node[fill,label={above:$u_6$},label={below:$\geq 2$}] (6) at (10,0) {};
\end{scope}

\begin{scope}[every edge/.style={draw=black}]
    \path (1) edge (6);
\end{scope}
\end{tikzpicture}
}
\end{center}
\end{figure}

\begin{figure}[htbp]
\begin{center}
\subfigure[\label{234422}$L(u_1)\subseteq L(u_2)$ and $L(u''_3)\subseteq L(u'_3)$.]{
\begin{tikzpicture}[scale=0.7]{thick}
\begin{scope}[every node/.style={circle,draw,minimum size=1pt,inner sep=2}]
    \node[fill,label={above:$u_1$},label={below:$\geq 2$}] (1) at (0,0) {};
    \node[fill,label={above:$u_2$},label={below:$\geq 3$}] (2) at (2,0) {};
    \node[fill,label={[label distance = +4pt]above right:$u_3$},label={below:$\geq 4$}] (3) at (4,0) {};
    \node[fill,label={above:$u_4$},label={below:$\geq 4$}] (4) at (6,0) {};
    \node[fill,label={above:$u_5$},label={below:$\geq 2$}] (5) at (8,0) {};
    \node[fill,label={above:$u_6$},label={below:$\geq 2$}] (6) at (10,0) {};
    
    \node[fill,label={[label distance = +4pt]above right:$u'_3$},label={below left:$\geq 3$}] (3') at (4,2) {};
    \node[fill,label={above:$u''_3$},label={below left:$\geq 2$}] (3'') at (4,4) {};
\end{scope}

\begin{scope}[every edge/.style={draw=black}]
    \path (1) edge (6);
    \path (3) edge (3'');
\end{scope}
\end{tikzpicture}
}
\hfil
\subfigure[\label{2334422}]{
\begin{tikzpicture}[scale=0.7]{thick}
\begin{scope}[every node/.style={circle,draw,minimum size=1pt,inner sep=2}]
    \node[fill,label={above:$u_1$},label={below:$\geq 2$}] (1) at (0,0) {};
    \node[fill,label={above:$u_2$},label={below:$\geq 3$}] (2) at (2,0) {};
    \node[fill,label={[label distance = +4pt]above right:$u_3$},label={below:$\geq 3$}] (3) at (4,0) {};
    \node[fill,label={above:$u_4$},label={below:$\geq 4$}] (4) at (6,0) {};
    \node[fill,label={above:$u_5$},label={below:$\geq 4$}] (5) at (8,0) {};
    \node[fill,label={above:$u_6$},label={below:$\geq 2$}] (6) at (10,0) {};
    \node[fill,label={above:$u_7$},label={below:$\geq 2$}] (7) at (12,0) {};
    
    \node[fill,label={above:$u'_3$},label={below left:$\geq 2$}] (3') at (4,2) {};
\end{scope}

\begin{scope}[every edge/.style={draw=black}]
    \path (1) edge (7);
    \path (3) edge (3');
\end{scope}
\end{tikzpicture}
}
\hfil
\subfigure[\label{2244343}$L(u'_5)\subseteq L(u_5)$.]{
\begin{tikzpicture}[scale=0.7]{thick}
\begin{scope}[every node/.style={circle,draw,minimum size=1pt,inner sep=2}]
    \node[fill,label={above:$u_1$},label={below:$\geq 2$}] (1) at (0,0) {};
    \node[fill,label={above:$u_2$},label={below:$\geq 2$}] (2) at (2,0) {};
    \node[fill,label={above:$u_3$},label={below:$\geq 4$}] (3) at (4,0) {};
    \node[fill,label={above:$u_4$},label={below:$\geq 4$}] (4) at (6,0) {};
    \node[fill,label={[label distance = +4pt]above right:$u_5$},label={below:$\geq 3$}] (5) at (8,0) {};
    \node[fill,label={[label distance = +4pt]above right:$u_6$},label={below:$\geq 4$}] (6) at (10,0) {};
    \node[fill,label={above:$u_7$},label={below:$\geq 3$}] (7) at (12,0) {};
    
    \node[fill,label={above:$u'_5$},label={below left:$\geq 2$}] (5') at (8,2) {};
    \node[fill,label={above:$u'_6$},label={below left:$\geq 3$}] (6') at (10,2) {};
\end{scope}

\begin{scope}[every edge/.style={draw=black}]
    \path (1) edge (7);
    \path (5) edge (5');
    \path (6) edge (6');
\end{scope}
\end{tikzpicture}
}
\caption{Colorable graphs.}
\label{colorable figure}
\end{center}
\end{figure}

\begin{proof}
In the following proofs, whenever the size of a list $|L(u)|\geq i$, we assume that $|L(u)|=i$ by removing the extra colors from the list while preserving the inclusions.

\begin{itemize}
\item[(i)] If $L(u_1)=L(u_2)$, then we color $u_3$ with a color in $L(u_3)\setminus L(u_2)$, followed by $u_4$, $u_2$, and $u_1$ in this order. If $L(u_1)\neq L(u_2)$, then we color $u_2$ with a color in $L(u_2)\setminus L(u_1)$, followed by $u_4$, $u_3$, and $u_1$ in this order.

\item[(ii)] Since $|L(u_1)|\geq 2$, $|L(u'_3)|\geq 3$, and both $L(u_1)$ and $L(u'_3)$ are contained in $\{a,b,c,d\}$, we have a color $x\in L(u_1)\cap L(u'_3)$ by the pigeonhole principle. We color $u_1$ and $u'_3$ with $x$, then $u'_2$, $u_2$, $u_3$, and $u_4$ are colorable by \Cref{2232}. 

\item[(iii)] We restrict $u_2$ by one color. Then, we color $u_2$ and $u_1$ in this order first. By \Cref{restriction lemma}, we color $u_3$, $u_4$, and $u_5$ last. 

\item[(iv)] If $L(u_1)$ and $L(u'_3)$ share a common color $x$, then we color $u_1$ and $u'_3$ with $x$. The remaining vertices $u_5$, $u_4$, $u_3$, and $u_2$ are colorable by \Cref{2232}. So, $L(u_1)\cap L(u'_3)=\emptyset$ and by symmetry, we also have $L(u_5)\cap L(u'_3)=\emptyset$. 

W.l.o.g. we set $L(u'_3)=\{a,b\}$. As a result, $L(u_1)=L(u_5)=\{c,d\}$. Recall that $L(u_1)\subseteq L(u_2)$ and $L(u_5)\subseteq L(u_4)$. So, we can color $u_1$ and $u_4$ with $c$, $u_2$ and $u_5$ with $d$, $u_3$ with $a$, and $u'_3$ with $b$.

\item[(v)] We have $L(u'_3)\subseteq L(u_3)$. Otherwise, we can color $u'_3$ with a color in $L(u'_3)\setminus L(u_3)$, then $u_1$, $u_2$, $u_3$, $u_4$, and $u_5$ are colorable by \Cref{22332}.

If $L(u_1)$ and $L(u'_3)$ share a common color $x$, then we color $u_1$ and $u'_3$ with $x$, and $u_2$, $u_3$, $u_4$, and $u_5$ are colorable by \Cref{2232}. 

If $L(u_1)\cap L(u'_3)=\emptyset$, then w.l.o.g. we set $L(u_1)=\{a,b\}$ and $L(u'_3)=\{c,d\}$. Since $|L(u_2)|\geq 3$, w.l.o.g. we color $u_2$ with $a$ then $u_1$ with $b$. As both $L(u'_3)$ and $L(u_3)$ contain $\{c,d\}$, we still have $|L(u'_3)|\geq 2$ and $|L(u_3)|\geq 2$, thus $u'_3$, $u_3$, $u_4$, and $u_5$ are colorable by \Cref{2232}. 

\item[(vi)] By the pigeonhole principle, there exists $x\in L(u'_3)\cap L(u_5)$. If $x\notin L(u_2)$, then we color $u'_3$ and $u_5$ with $x$. The remaining vertices $u_1$, $u_2$, $u_3$, and $u_4$ are colorable by \Cref{2232}. So $x\in L(u_2)$.

We also have $L(u_1)=L(u_2)$. Otherwise, we color $u_2$ with a color in $L(u_2)\setminus L(u_1)$, then $u_5$, $u_4$, $u_3$, $u'_3$ are colorable by \Cref{2232}, and we finish by coloring $u_1$.

Since $x\in L(u'_3)\cap L(u_5)\cap L(u_2) \cap L(u_1)$, we color $u_1$, $u'_3$, and $u_5$ with $x$, then we color $u_2$, $u_4$, and $u_3$ in this order.

\item[(vii)] If there exists $x\in L(u_3)\setminus L(u_5)$, then we color $u_3$ with $x$, then $u'_3$, $u_1$, $u_2$, $u_4$, $u'_4$, and $u_5$ in this order.

If $L(u_3)=L(u_5)$, then we color $u_4$ with a color $y$ in $L(u_4)\setminus L(u_5)$. Recall that $L(u'_3)\subseteq L(u_3)$, so $y\notin L(u'_3)\cup L(u_3)\cup L(u_5)$. We color $u_1$, $u_2$, $u_3$, and $u'_3$ by \Cref{2232}. Finally, we finish by color $u'_4$ and $u_5$ in this order.

\item[(viii)] If there exists two same sets of colors between $L(u_2)$, $L(u_4)$, and $L(u'_3)$, say $L(u_2)=L(u_4)$, then we color $u_3$ with $x\in L(u_3)\setminus L(u_2)$. Recall that $L(u_1)\subseteq L(u_2)$ and $L(u_5)\subseteq L(u_4)$ so $x\notin L(u_1)\cup L(u_2)\cup L(u_4)\cup L(u_5)$. We finish by coloring $u''_3$, $u'_3$, $u_1$, $u_2$, $u_4$, $u_5$ in this order.

If $L(u_2)$, $L(u_4)$, and $L(u'_3)$ are all different, then we color the graph as follows. By the pigeonhole principle, two sets between $L(u_1)$, $L(u_5)$, and $L(u''_3)$ must share a common color, say $L(u_1)\cap L(u_5)\neq \emptyset$. In other words, $|L(u_1)\cup L(u_5)|\leq 3$. Then, we color $u_3$ with a color in $L(u_3)\setminus (L(u_1)\cup L(u_5))$. We color $u''_3$ and $u'_3$ in this order. Now, we can color $u_2$ and $u_4$ since they see the same two colors but initially $L(u_2)\neq L(u_4)$. Finally, we finish by coloring $u_1$ and $u_5$.

\item[(ix)] If $L(u_1)=L(u_2)$, then we restrict $L(u_3)$ to $L(u_3)\setminus L(u_2)$. We color $u_3$, $u_4$, $u_5$, $u_6$ by \Cref{2232}, then we finish by coloring $u_2$ and $u_1$ in this order.

If there exists $x\in L(u_1)\setminus L(u_2)$, then we color $u_1$ with $x$. Finally, $u_6$, $u_5$, $u_4$, $u_3$, and $u_2$ are colorable by \Cref{22332}.

\item[(x)] If $L(u_5)=L(u_6)$, then we restrict $L(u_4)$ to $L(u_4)\setminus L(u_5)$. Recall that we have $L(u_1)\subseteq L(u_2)$ and $L(u''_3)\subseteq L(u'_3)$. We can thus color $u_1$, $u_2$, $u_3$, $u_4$, $u'_3$, and $u''_3$ by \Cref{23432}. We finish by coloring $u_5$ and $u_6$ in this order.

If there exists $a\in L(u_6)\setminus L(u_5)$, then we color $u_6$ with $a$. Observe that $L(u_5)\subseteq \{b,c,d\} = L(u_4)$ after we color $u_6$ with $a$. Recall that we also have $L(u_1)\subseteq L(u_2)$ and $L(u''_3)\subseteq L(u'_3)$. So, we color the remaining vertices $u_1$, $u_2$, $u_3$, $u'_3$, $u''_3$, $u_4$, and $u_5$ by \Cref{23432bis} 

\item[(xi)] If $L(u_6)=L(u_7)$, then we restrict $L(u_5)$ to $L(u_5)\setminus L(u_6)$. We color $u_1$, $u_2$, $u_3$, $u'_3$, $u_4$, and $u_5$ by \Cref{23342}. We finish by coloring $u_6$ and $u_7$ in this order.

If there exists $ x\in L(u_7)\setminus L(u_6)$, then we color $u_7$ with $x$. We restrict $u_3$ by one color to color $u_4$, $u_5$, and $u_6$ last by \Cref{restriction lemma}. Then, $u'_3$, $u_3$, $u_2$, and $u_1$ are colorable by \Cref{2232}.

\item[(xii)] If $L(u_1)=L(u_2)$, then we restrict $L(u_3)$ to $L(u_3)\setminus L(u_2)$. We color $u_3$, $u_4$, $u_5$, $u'_5$, $u_6$, $u'_6$, and $u_7$ by \Cref{24343}. Then, we finish by coloring $u_2$ and $u_1$ in this order.

If there exists $x\in L(u_2)\setminus L(u_1)$, then we color $u_2$ with $x$. We color $u'_5$, $u_5$, $u_4$, $u_6$, $u'_6$, and $u_7$ by \Cref{2343}. Finally, we finish by coloring $u_3$ and $u_1$ in this order. 
\end{itemize}
 
\end{proof}

\subsection{Structural properties of \texorpdfstring{$G$}{G}\label{tutu}}

\begin{lemma}\label{connected}
Graph $G$ is connected.
\end{lemma}
\begin{proof}
Otherwise a component of $G$ would be a smaller counterexample.
\end{proof}

\begin{lemma}\label{minimumDegree}
The minimum degree of $G$ is at least 2.
\end{lemma}

\begin{proof}
By \Cref{connected}, the minimum degree is at least 1 or $G$ would be a single isolated vertex which is 4-colorable. If $G$ contains a degree 1 vertex $v$, then we can simply remove the unique edge incident to $v$ and 2-distance color the resulting graph, which is possible by minimality of $G$. Then, we add the edge back and color $v$ (at most 3 constraints and 4 colors).
\end{proof}

\begin{lemma}[\cite{bi12bis} Lemmas 10,11, and 12] \label{bi}
Graph $G$ has no:
\begin{enumerate}
\item[(i)] $6^+$-paths
\item[(ii)] $(1^+,4^+,5^+)$
\item[(iii)] $(2^+,3^+,4^+)$
\item[(iv)] $(3^+,3^+,3^+)$
\item[(v)] $(330-045)$
\item[(vi)] $(431-133)$
\end{enumerate}
\end{lemma}

\begin{proof}
The proofs of the reducibility of these configurations are presented in \citep{bi12bis} with the same notations. These configurations were reduced for planar subcubic graphs of girth at least 22 where all 3-vertices and 2-vertices on the incident paths are distinct, but the same proofs hold for $G$ since the girth is still high enough for all vertices to remain distinct.
\end{proof}

The following configurations are new or stronger versions of configurations in \citep{bi12bis}.

\begin{lemma}\label{reducible pair}
Graph $G$ cannot contain the following pairs:
\begin{enumerate}
\item[(i)] $(430-024)$
\item[(ii)] $(540-014)$
\item[(iii)] $(431-114)$
\item[(iv)] $(422-223)$
\item[(v)] $(422-214)$
\item[(vi)] $(412-233)$
\item[(vii)] $(332-233)$
\end{enumerate}
\end{lemma}

\begin{figure}[htbp]
\begin{center}
\begin{tikzpicture}[scale=0.7]{thick}
\begin{scope}[every node/.style={circle,draw,minimum size=1pt,inner sep=2}]
	\node[label={above:$u'_{k+1}$}] (0) at (-2,0) {};
    \node[fill,label={above:$u'_k$}] (1) at (0,0) {};
    \node[fill,label={above:$u'_1$}] (2) at (2,0) {};
    \node[fill,label={[label distance = +4pt]above right:$u$}] (3) at (4,0) {};
    \node[fill,label={above:$v'_m$}] (4) at (6,0) {};
    \node[fill,label={above:$v'_1$}] (5) at (8,0) {};
    \node[fill,label={[label distance = +4pt]above right:$v$}] (6) at (10,0) {};
    \node[fill,label={above:$v'_{m+1}$}] (7) at (12,0) {};
    \node[fill,label={above:$v'_{m+p}$}] (8) at (14,0) {};
    \node[label={above:$v'_{m+p+1}$}] (9) at (16,0) {};
    
    \node[fill,label={[label distance = +4pt]above right:$u_1$}] (3') at (4,2) {};
    \node[fill,label={[label distance = +4pt]above right:$u_l$}] (3'') at (4,4) {};
    \node[label={above:$u_{l+1}$}] (3''') at (4,6) {};
    
    \node[fill,label={[label distance = +4pt]above right:$v_1$}] (6') at (10,2) {};
    \node[fill,label={[label distance = +4pt]above right:$v_n$}] (6'') at (10,4) {};
    \node[label={above:$v_{n+1}$}] (6''') at (10,6) {};
\end{scope}

\begin{scope}[every edge/.style={draw=black}]
    \path (0) edge (1);
    \path (1) edge[dashed] (2);
    \path (2) edge (3);
    
    \path (3) edge (3');
    \path (3') edge[dashed] (3'');
    \path (3'') edge (3''');
    
    \path (3) edge (4);
    \path (4) edge[dashed] (5);
    \path (5) edge (6);
    
    \path (6) edge (6');
    \path (6') edge[dashed] (6'');
    \path (6'') edge (6''');
    
    \path (6) edge (7);
    \path (7) edge[dashed] (8);
    \path (8) edge (9);
\end{scope}
\end{tikzpicture}
\caption{\Cref{reducible pair} notations.}
\end{center}
\end{figure}

\begin{proof}
First, we define the following notations:
\begin{itemize}
\item Let $u=(k^+,l^+,m)$ and $v=(m,n^+,p^+)$ form the pair $(klm-mnp)$.
\item Let $uu'_1u'_2\dots u'_{k+1}$ be the $k^+$-path incident to $u$.
\item Let $uu_1u_2\dots u_{l+1}$ be the $l^+$-path incident to $u$.
\item Let $vv'_1v'_2\dots v'_mu$ be the $m$-path incident to $u$ and $v$.
\item Let $vv_1v_2\dots v_{n+1}$ be the $n^+$-path incident to $v$.
\item Let $vv'_{m+1}v'_{m+2}\dots v'_{m+p+1}$ be the $p^+$-path incident to $v$.
\item For every pair $(klm-mnp)$ from (i) to (vii), we define the subgraph 

$H=\{u,v,u'_1,u'_2,\dots, u'_{k-1},u_1,u_2,\dots,u_{l-1},v'_1,v'_2,\dots,v'_{m+p-1},v_1,v_2,\dots,v_{n-1}\}$.  
\end{itemize}

First, observe that all vertices in $H$ are distinct since $G$ has girth at least 21. In the following proofs, we will always color $G-H$ first, which is possible by minimality of $G$. For each vertex of $H$, its list of available colors will always be $\{a,b,c,d\}$ from which we removed the colors it sees on its neighbors from $G-H$. Then, we will show that the coloring of $G-H$ is extendable to $H$ using colorable graphs from \Cref{colorable lemma}. For convenience, we will cite \Cref{colorable figure} from now on.

Also observe that when two adjacent vertices $x_1$, $x_2$ in $H$ sees a common color with $|L(x_1)|\leq |L(x_2)|$, then $L(x_1)\subseteq L(x_2)$. This simple remark will be used throughout the proofs, mostly to justify the use of \Cref{colorable figure}(iv), (vii), (viii), (x), and (xii). For conciseness, we will state the inclusions directly when needed.

\begin{itemize}
\item[(i)] We restrict $u$ by one color to color $u'_1$, $u'_2$, and $u'_3$ last by \Cref{restriction lemma}. We restrict $v$ by one color to color $v'_1$, $v'_2$, and $v'_3$ afterwards. Finally, $v_1$, $v$, $u$, $u_1$, and $u_2$ are colorable by \Cref{22332}. 

\item[(ii)] We restrict $v$ by one color to color $v'_1$, $v'_2$, and $v'_3$ last. We restrict $u'_1$ by one color to color $u'_2$, $u'_3$, and $u'_4$ afterwards. Finally, we color $v$, then $u'_1$, $u$, $u_1$, $u_2$, and $u_3$ are colorable by \Cref{22332}.

\item[(iii)] We restrict $u$ by one color to color $u'_1$, $u'_2$, and $u'_3$ last. We restrict $v$ by one color to color $v'_2$, $v'_3$, and $v'_4$ afterwards. Finally, we color $v$, then $v'_1$, $u$, $u_1$, and $u_2$ are colorable by \Cref{2232}.

\item[(iv)] We restrict $u$ by one color to color $u'_1$, $u'_2$, and $u'_3$ last. Then, $v'_4$, $v'_3$, $v$, $v_1$, $v'_1$, $v'_2$, $u$, and $u_1$ are colorable by \Cref{2334422}. 

\item[(v)] We restrict $v$ by one color to color $v'_3$, $v'_4$, and $v'_5$ last. We restrict $u$ by one color to color $u'_1$, $u'_2$, and $u'_3$ afterwards. Finally, we color $v$, then $u_1$, $u$, $v'_2$, and $v'_1$ are colorable by \Cref{2232}.

\item[(vi)] We restrict $u$ by one color to color $u'_1$, $u'_2$, and $u'_3$ last. Then, we color $u$ and observe that since $L(v'_2)\subseteq L(v'_1)$, $L(v_2)\subseteq L(v_1)$, and $L(v'_4)\subseteq L(v'_3)$, $v'_2$, $v'_1$, $v$, $v_1$, $v_2$, $v'_3$, and $v'_4$ are colorable by \Cref{23432bis}.

\item[(vii)] We color $v$ with $x\in L(v)\setminus L(v_1)$. Observe that $L(v_2)\subseteq L(v_1)$ so $x\notin L(v_1)\cup L(v_2)$. Then, we color $v'_4$, and $v'_3$ in this order. Since $L(u'_2)\subseteq L(u'_1)$, $L(u_2)\subseteq L(u_1)$, and $L(v'_1)\subseteq L(v'_2)$, $u'_2$, $u'_1$, $u$, $u_1$, $u_2$, $v'_2$ and $v'_1$ are colorable by \Cref{23432bis}. Finally, we finish by coloring $v_1$ and $v_2$.
\end{itemize}

\end{proof}

\begin{lemma}\label{reducible triple}
Graph $G$ cannot contain the following triples:
\begin{enumerate}
\item[(i)] $(550-020-045)$
\item[(ii)] $(440-040-024)$
\item[(iii)] $(550-021-134)$
\item[(iv)] $(420-031-134)$
\item[(v)] $(550-022-224)$
\item[(vi)] $(540-032-214)$
\item[(vii)] $(540-032-233)$
\item[(viii)] $(420-042-214)$
\item[(ix)] $(420-042-233)$
\item[(x)] $(431-131-124)$
\item[(xi)] $(421-141-124)$
\item[(xii)] $(431-112-224)$
\item[(xiii)] $(421-132-233)$
\item[(xiv)] $(421-132-214)$
\item[(xv)] $(422-222-214)$
\item[(xvi)] $(332-222-224)$
\item[(xvii)] $(332-232-233)$
\item[(xviii)] $(332-232-214)$
\item[(xix)] $(412-232-214)$
\end{enumerate}
\end{lemma}

\begin{figure}[htbp]
\begin{center}
\begin{tikzpicture}[scale=0.55]{thick}
\begin{scope}[every node/.style={circle,draw,minimum size=1pt,inner sep=2}]
	\node[label={above:$u'_{k+1}$}] (0) at (-2,0) {};
    \node[fill,label={above:$u'_k$}] (1) at (0,0) {};
    \node[fill,label={above:$u'_1$}] (2) at (2,0) {};
    \node[fill,label={[label distance = +4pt]above right:$u$}] (3) at (4,0) {};
    \node[fill,label={above:$v'_m$}] (4) at (6,0) {};
    \node[fill,label={above:$v'_1$}] (5) at (8,0) {};
    \node[fill,label={[label distance = +4pt]above right:$v$}] (6) at (10,0) {};
    \node[fill,label={above:$v'_{m+1}$}] (7) at (12,0) {};
    \node[fill,label={above:$v'_{m+p}$}] (8) at (14,0) {};
    \node[fill,label={[label distance = +4pt]above right:$w$}] (9) at (16,0) {};
    \node[fill,label={above:$w'_1$}] (10) at (18,0) {};
    \node[fill,label={above:$w'_r$}] (11) at (20,0) {};
    \node[label={above:$w'_{r+1}$}] (12) at (22,0) {};
    
    \node[fill,label={[label distance = +4pt]above right:$u_1$}] (3') at (4,2) {};
    \node[fill,label={[label distance = +4pt]above right:$u_l$}] (3'') at (4,4) {};
    \node[label={above:$u_{l+1}$}] (3''') at (4,6) {};
    
    \node[fill,label={[label distance = +4pt]above right:$v_1$}] (6') at (10,2) {};
    \node[fill,label={[label distance = +4pt]above right:$v_n$}] (6'') at (10,4) {};
    \node[label={above:$v_{n+1}$}] (6''') at (10,6) {};
    
    \node[fill,label={[label distance = +4pt]above right:$w_1$}] (9') at (16,2) {};
    \node[fill,label={[label distance = +4pt]above right:$w_q$}] (9'') at (16,4) {};
    \node[label={above:$w_{q+1}$}] (9''') at (16,6) {};
\end{scope}

\begin{scope}[every edge/.style={draw=black}]
    \path (0) edge (1);
    \path (1) edge[dashed] (2);
    \path (2) edge (3);
    
    \path (3) edge (3');
    \path (3') edge[dashed] (3'');
    \path (3'') edge (3''');
    
    \path (3) edge (4);
    \path (4) edge[dashed] (5);
    \path (5) edge (6);
    
    \path (6) edge (6');
    \path (6') edge[dashed] (6'');
    \path (6'') edge (6''');
    
    \path (6) edge (7);
    \path (7) edge[dashed] (8);
    \path (8) edge (9);
    
    \path (9) edge (9');
    \path (9') edge[dashed] (9'');
    \path (9'') edge (9''');
    
    \path (9) edge (10);
    \path (10) edge[dashed] (11);
    \path (11) edge (12);
\end{scope}
\end{tikzpicture}
\caption{\Cref{reducible triple} notations.}
\end{center}
\end{figure}

\begin{proof}
We will use similar notations to the proofs of \Cref{reducible pair}:
\begin{itemize}
\item Let $u=(k^+,l^+,m)$, $v=(m,n^+,p)$, and $w=(p,q^+,r^+)$ form the triple $(klm-mnp-pqr)$.
\item Let $uu'_1u'_2\dots u'_{k+1}$ be the $k^+$-path incident to $u$.
\item Let $uu_1u_2\dots u_{l+1}$ be the $l^+$-path incident to $u$.
\item Let $vv'_1v'_2\dots v'_mu$ be the $m$-path incident to $u$ and $v$.
\item Let $vv_1v_2\dots v_{n+1}$ be the $n^+$-path incident to $v$.
\item Let $vv'_{m+1}v'_{m+2}\dots v'_{m+p}w$ be the $p$-path incident to $v$ and $w$.
\item Let $ww_1w_2\dots w_{q+1}$ be the $q^+$-path incident to $w$.
\item Let $ww'_1w'_2\dots w'_{r+1}$ be the $r^+$-path incident to $w$.
\item For every triple $(klm-mnp-pqr)$ from (i) to (xix), we define the subgraph 

$H=\{u,v,w,u'_1,u'_2,\dots, u'_{k-1},$ $u_1,u_2,\dots,u_{l-1},$ $v'_1,v'_2,\dots,v'_{m+p},$ $v_1,v_2,\dots,v_{n-1},$\\ $w_1,w_2,\dots,w_{q-1},$ $w'_1,w'_2,\dots, w'_{r-1}\}$.
\end{itemize}

Similarly, all vertices in $H$ are distinct since $G$ has girth at least 21. We will color $G-H$ by minimality of $G$ first, then extend that coloring to $H$ using \Cref{colorable figure}.

\begin{itemize}
\item[(i)] We restrict $u'_1$ by one color to color $u'_2$, $u'_3$, and $u'_4$ last. We restrict $u_1$ by one color to color $u_2$, $u_3$, and $u_4$ afterwards. We restrict $w$ by one color to color $w_1$, $w_2$, and $w_3$ afterwards. We restrict $w'_1$ by one color to color $w'_2$, $w'_3$, and $w'_4$ afterwards. Now, we color $v_1$, $v$, $w$, $u$, $u_1$, and $u'_1$ by \Cref{2343}. Then, we color $w'_1$.

\item[(ii)] We restrict $u$ by one color to color $u'_1$, $u'_2$, and $u'_3$ last. We restrict $u$ again by one color to color $u_1$, $u_2$, and $u_3$ afterwards. We restrict $v$ by one color to color $v_1$, $v_2$, and $v_3$ afterwards. We restrict $w$ by one color to color $w'_1$, $w'_2$, and $w'_3$ afterwards. We color the remaining vertices $w_1$, $w$, $v$, and $u$ by \Cref{2232}.

\item[(iii)] We restrict $u'_1$ by one color to color $u'_2$, $u'_3$, and $u'_4$ last. We restrict $u_1$ by one color to color $u_2$, $u_3$, and $u_4$ afterwards. We restrict $w$ by one color to color $w'_1$, $w'_2$, and $w'_3$ afterwards. We restrict $v'_1$ by one color to color $w$, $w_1$, and $w_2$ afterwards. The remaining vertices $v_1$, $v$, $v'_1$, $u$, $u_1$, and $u'_1$ are colorable by \Cref{2343}.

\item[(iv)] We restrict $u$ by one color to color $u'_1$, $u'_2$, and $u'_3$ last. We restrict $w$ by one color to color $w'_1$, $w'_2$, and $w'_3$ afterwards. We restrict $v'_1$ by one color to color $w$, $w_1$, and $w_2$ afterwards. The remaining vertices $u_1$, $u$, $v$, $v'_1$, $v_1$, and $v_2$ are colorable by \Cref{22432}.

\item[(v)] We restrict $u'_1$ by one color to color $u'_2$, $u'_3$, and $u'_4$ last. We restrict $u_1$ by one color to color $u_2$, $u_3$, and $u_4$ afterwards. We restrict $w$ by one color to color $w'_1$, $w'_2$, and $w'_3$ afterwards. The remaining vertices $w_1$, $w$, $v'_2$, $v'_1$, $v$, $v_1$, $u$, $u'_1$, and $u_1$ are colorable by \Cref{2244343} as $L(v_1)\subseteq L(v)$.

\item[(vi)] We restrict $u'_1$ by one color to color $u'_2$, $u'_3$, and $u'_4$ last. We restrict $w$ by one color to color $w'_1$, $w'_2$, and $w'_3$ afterwards. We color $v$ with $x\in L(v)\setminus L(v_1)$. Observe that $L(v_2)\subseteq L(v_1)$ so $x\notin L(v_1)\cup L(v_2)$. Now, we color $w$, $v'_2$, $v'_1$ in this order. The vertices $u'_1$, $u$, $u_1$, $u_2$, and $u_3$ are colorable by \Cref{22332}. Then, we color the remaining vertices $v_1$ and $v_2$ in this order.

\item[(vii)] We restrict $u'_1$ by one color to color $u'_2$, $u'_3$, and $u'_4$ last. We color $w$ with $x\in L(w)\setminus L(w_1)$. Observe that $L(w_2)\subseteq L(w_1)$ so $x\notin L(w_1)\cup L(w_2)$. We color $v$ with $y\in L(v)\setminus L(v_1)$. Observe that $L(v_2)\subseteq L(v_1)$ so $y\notin L(v_1)\cup L(v_2)$. Now, we color $w'_2$, $w'_1$, $v'_2$, $v'_1$ in this order. The vertices $u'_1$, $u$, $u_1$, $u_2$, and $u_3$ are colorable by \Cref{22332}. Then, we color the remaining vertices $v_1$, $v_2$, $w_1$, and $w_2$ in this order.

\item[(viii)] We restrict $u$ by one color to color $u'_1$, $u'_2$, and $u'_3$ last. We restrict $w$ by one color to color $w'_1$, $w'_2$, and $w'_3$ afterwards. We restrict $v$ by one color to color $v_1$, $v_2$, and $v_3$ afterwards. We color $w$ then the remaining vertices $u_1$, $u$, $v$, $v'_1$, and $v'_2$ are colorable by \Cref{22332}.

\item[(ix)] We restrict $u$ by one color to color $u'_1$, $u'_2$, and $u'_3$ last. We restrict $v$ by one color to color $v_1$, $v_2$, and $v_3$ afterwards. We color $w$ with $x\in L(w)\setminus L(w_1)$. Observe that $L(w_2)\subseteq L(w_1)$ so $x\notin L(w_1)\cup L(w_2)$.  We color $w'_2$ and $w'_1$ in this order. The vertices $u_1$, $u$, $v$, $v'_1$, and $v'_2$ are colorable by \Cref{22332}. Now, we color the remaining vertices $w_1$ and $w_2$ in this order. 

\item[(x)] We restrict $u$ by one color to color $u'_1$, $u'_2$, and $u'_3$ last. We restrict $v'_1$ by one color to color $u$, $u_1$, and $u_2$ afterwards. We restrict $w$ by one color to color $w'_1$, $w'_2$, and $w'_3$ afterwards. We restrict $L(v'_2)$ to $L(v'_2)\setminus L(w_1)$. We color the vertices $w$, $v'_2$, $v$, $v'_1$, $v_1$ and $v_2$ by \Cref{22432}. Then, we color the remaining vertex $w_1$.

\item[(xi)] We restrict $u$ by one color to color $u'_1$, $u'_2$, and $u'_3$ last. We restrict $v$ by one color to color $v_1$, $v_2$, and $v_3$ afterwards. We restrict $w$ by one color to color $w'_1$, $w'_2$, and $w'_3$ afterwards. We restrict $L(v'_1)$ to $L(v'_1)\setminus L(u_1)$. We color $w_1$, $w$, $v'_2$, $v$, $v'_1$, and $u$ by \Cref{224322}. Then, we color the remaining vertex $u_1$.

\item[(xii)] We restrict $u$ by one color to color $u'_1$, $u'_2$, and $u'_3$ last. We restrict $v'_1$ by one color to color $u$, $u_1$, and $u_2$ afterwards. We restrict $w$ by one color to color $w'_1$, $w'_2$, and $w'_3$ afterwards. We color the remaining vertices $w_1$, $w$, $v'_3$, $v'_2$, $v$, and $v'_1$ by \Cref{224322}.

\item[(xiii)] We restrict $u$ by one color to color $u'_1$, $u'_2$, and $u'_3$ last. We color $w$ with $x\in L(w)\setminus L(w_1)$. Observe that $L(w_2)\subseteq L(w_1)$ so $x\notin L(w_1)\cup L(w_2)$.  We color $w'_2$ and $w'_1$ in this order. The vertices $v_2$, $v_1$, $v$, $v'_3$, $v'_2$, $v'_1$, $u$, and $u_1$ are colorable by \Cref{234422} as $L(v_2)\subseteq L(v_1)$ and $L(v'_3)\subseteq L(v'_2)$. Now, we color the remaining vertices $w_1$ and $w_2$ in this order.

\item[(xiv)] We restrict $u$ by one color to color $u'_1$, $u'_2$, and $u'_3$ last. We restrict $w$ by one color to color $w'_1$, $w'_2$, and $w'_3$ afterwards. We color $w$ then the remaining vertices $v'_3$, $v'_2$, $v$,  $v_1$, $v_2$, $v'_1$, $u$, and $u_1$ are colorable by \Cref{234422} as $L(v_2)\subseteq L(v_1)$ and $L(v'_3)\subseteq L(v'_2)$.

\item[(xv)] We restrict $u$ by one color to color $u'_1$, $u'_2$, and $u'_3$ last. We restrict $w$ by one color to color $w'_1$, $w'_2$, and $w'_3$ afterwards. We color $w$ then the remaining vertices $v'_4$, $v'_3$, $v$,  $v_1$, $v'_1$, $v'_2$, $u$, and $u_1$ are colorable by \Cref{2334422}.

\item[(xvi)] We restrict $w$ by one color to color $w'_1$, $w'_2$, and $w'_3$ last. We color $u$ with $x\in L(u)\setminus L(u_1)$. Observe that $L(u_2)\subseteq L(u_1)$ so $x\notin L(u_1)\cup L(u_2)$.  We color $u'_2$ and $u'_1$ in this order. We color $v'_2$, $v'_1$, $v$,  $v_1$, $v'_3$, $v'_4$, $w$, and $w_1$ by \Cref{2334422}. Now, we color the remaining vertices $u_1$ and $u_2$ in this order. 

\item[(xvii)] We color $w$ with $x\in L(w)\setminus L(w_1)$. Observe that $L(w_2)\subseteq L(w_1)$ so $x\notin L(w_1)\cup L(w_2)$. We color $v$ with $y\in L(v)\setminus L(v_1)$. Observe that $L(v_2)\subseteq L(v_1)$ so $y\notin L(v_1)\cup L(v_2)$.  We color $w'_2$, $w'_1$, $v'_4$, and $v'_3$ in this order. We color $v'_1$, $v'_2$, $u$,  $u_1$, $u_2$, $u'_1$, $u'_2$ by \Cref{23432bis} as $L(u'_2)\subseteq L(u'_1)$, $L(u_2)\subseteq L(u_1)$, and $L(v'_1)\subseteq L(v'_2)$. Now, we color the remaining vertices $v_1$, $v_2$, $w_1$ and $w_2$ in this order.

\item[(xviii)] We restrict $w$ by one color to color $w'_1$, $w'_2$, and $w'_3$ last. We color $v$ with $x\in L(v)\setminus L(v_1)$. Observe that $L(v_2)\subseteq L(v_1)$ so $x\notin L(v_1)\cup L(v_2)$. We color $w$,  $v'_4$, and $v'_3$ in this order. We color $v'_1$, $v'_2$, $u$,  $u_1$, $u_2$, $u'_1$, $u'_2$ by \Cref{23432bis} as $L(u'_2)\subseteq L(u'_1)$, $L(u_2)\subseteq L(u_1)$, and $L(v'_1)\subseteq L(v'_2)$. Now, we color the remaining vertices $v_1$ and $v_2$ in this order.

\item[(xix)] We restrict $w$ by one color to color $w'_1$, $w'_2$, and $w'_3$ last. We restrict $u$ by one color to color $u'_1$, $u'_2$, and $u'_3$ afterwards. We color $u$ and $w$ then the remaining vertices $v'_2$, $v'_1$, $v$,  $v_1$, $v_2$, $v'_3$, $v'_4$ are colorable by \Cref{23432bis} as $L(v'_2)\subseteq L(v'_1)$, $L(v_2)\subseteq L(v_1)$, and $L(v'_4)\subseteq L(v'_3)$.
\end{itemize}

\end{proof}

\subsection{Discharging rules \label{tonton}}

In this section, we will define a discharging procedure that contradicts the structural properties of $G$ (see \Cref{bi,reducible pair,reducible triple}) showing that $G$ does not exist. We assign to each vertex $u$ the charge $\mu(u)=\frac{19}2d(u)-21$ and to each face $f$ the charge $\mu(f)=d(f)-21$. By \Cref{equation}, the total sum of the charges is negative. We then apply the following discharging rules:

\medskip

Let $u$ and $v$ be endvertices of a $m$-path where $u=(k,l,m)$ with $k+l+m\leq 7$ and $v=(m,n,p)$. Vertex $u$ gives charge to $v$ in the following cases:

\begin{itemize}
\item[\ru0] If $m=0$
\begin{itemize}
\item[(i)] and $v=(0,5,5)$, then $u$ gives $\frac52$ to $v$.
\item[(ii)] and $v=(0,4,5)$, then $u$ gives $\frac32$ to $v$.
\item[(iii)] and $v\in\{(0,3,5),(0,4,4)\}$, then $u$ gives $\frac12$ to $v$.
\item[(iv)] and $v=(0,2,5)$, then $u$ gives $\frac14$ to $v$.
\end{itemize}

\item[\ru1] If $m=1$
\begin{itemize}
\item[(i)] and $v\in\{(1,3,5),(1,4,4)\}$, then $u$ gives $\frac32$ to $v$.
\item[(ii)] and $v\in\{(1,3,4),(1,2,5)\}$, then $u$ gives $\frac12$ to $v$.
\end{itemize}

\item[\ru2] If $m=2$
\begin{itemize}
\item[(i)] and $v=(2,2,5)$, then $u$ gives $\frac34$ to $v$.
\item[(ii)] and $v\in\{(2,3,3),(2,1,5)\}$, then $u$ gives $\frac12$ to $v$.
\item[(iii)] and $v=(2,2,4)$, then $u$ gives $\frac14$ to $v$.
\end{itemize}

\item[\ru3] Finally, every $3$-vertex gives 1 to each $2$-vertex on its incident paths.

\end{itemize}


\subsection{Verifying that charges on each face and each vertex are non-negative} \label{verification}

Let $\mu^*$ be the assigned charges after the discharging procedure. In what follows, we will prove that: $$\forall u \in V(G), \mu^*(u)\ge 0 \text{ and } \forall f \in F(G), \mu^*(f)\ge 0.$$

First of all, since $G$ is connected (\Cref{connected}), has minimum degree at least 2 (\Cref{minimumDegree}), has girth at least 21, and the discharging rules do not interfere with charge on faces, every face $f$ verifies $\mu^*(f) = \mu(f) = d(f)-21\geq 0$.

Now, let $u$ be a vertex in $V(G)$. If $d(u) = 2$, then $u$ receives charge 1 from each endvertex of the path it lies on by \ru3; thus we get $\mu^*(u) = \mu(u) + 2\cdot 1 = \frac{19}2\cdot 2 -21 + 2 = 0$.

From now on, suppose that $d(u) = 3$ and let $u=(k,l,m)$. Recall that $\mu(u)=\frac{19}{2}\cdot 3 - 21 = \frac{15}{2}$:

\textbf{Case 1:} Suppose that $k+l+m\geq 8$. \\
First, observe that $u$ only gives away charges by \ru3. More precisely, $u$ gives a total of $k+l+m$ to $2$-vertices. Since there are no $6^+$-paths, $(1^+,4^+,5^+)$, $(2^+,3^+,4^+)$, or $(3^+,3^+,3^+)$ due to \Cref{bi}, then the only possible values for $k$, $l$, and $m$ are as follows:
\begin{itemize}
\item If $u$ is a $(5,5,0)$, $(5,4,0)$, $(5,3,0)$ or $(4,4,0)$, then $u$ cannot be adjacent to a vertex $v=(m,n,p)$ with $m+n+p\geq 8$ as $(430-024)$ is reducible by \Cref{reducible pair}(i). As a result, $u$ receives charge $\frac52$ (resp. $\frac32$, $\frac12$, or $\frac12$) by \ru0(i) (resp. \ru0(ii), \ru0(iii), or \ru0(iii)) when it is a $(5,5,0)$ (resp. $(5,4,0)$, $(5,3,0)$, or $(4,4,0)$). To sum up, we have
\begin{align*}
\mu^*(u) & = \frac{15}{2} + \frac52 - 5 - 5 = 0 & \text{when }u=(5,5,0)\\
& = \frac{15}{2} + \frac32 - 5 - 4 = 0 & \text{when }u=(5,4,0)\\
& = \frac{15}{2} + \frac12 - 5 - 3 = 0 & \text{when }u=(5,3,0)\\
& = \frac{15}{2} + \frac12 - 4 - 4 = 0 & \text{when }u=(4,4,0)\\  
\end{align*}

\item If $u$ is a $(5,3,1)$, $(4,4,1)$, or $(4,3,1)$, then $u$ cannot share a $1$-path with a vertex $v=(m,n,p)$ with $m+n+p\geq 8$ as $(431-114)$ is reducible by \Cref{reducible pair}(iii). As a result, $u$ receives charge $\frac32$ (resp. $\frac32$, or $\frac12$) by \ru1(i) (resp. \ru1(i), or \ru1(ii)) when it is a $(5,3,1)$ (resp. $(4,4,1)$, or $(4,3,1)$). To sum up, we have
\begin{align*}
\mu^*(u) & = \frac{15}{2} + \frac32 - 5 - 3 - 1 = 0 & \text{when }u=(5,3,1)\\
& = \frac{15}{2} + \frac32 - 4 - 4 - 1 = 0 & \text{when }u=(4,4,1)\\
& = \frac{15}{2} + \frac12 - 4 - 3 - 1 = 0 & \text{when }u=(4,3,1)\\  
\end{align*}

\item If $u$ is a $(5,2,2)$ or $(4,2,2)$, then $u$ cannot share a $2$-path with a vertex $v=(m,n,p)$ with $m+n+p\geq 8$ as $(422-223)$ and $(422-214)$ are reducible respectively by \Cref{reducible pair}(iv) and \Cref{reducible pair}(v). As a result, $u$ receives charge $\frac34$ (resp. $\frac14$) by \ru2(i) (resp. \ru2(iii)) when it is a $(5,2,2)$ (resp. $(4,2,2)$) twice (once from each incident 2-path). To sum up, we have
\begin{align*}
\mu^*(u) & = \frac{15}{2} + 2\cdot\frac34 - 5 - 2 - 2 = 0 & \text{when }u=(5,2,2)\\
& = \frac{15}{2} + 2\cdot\frac14 - 4 - 2 - 2 = 0 & \text{when }u=(4,2,2)\\
\end{align*}

\item If $u$ is a $(3,3,2)$, then $u$ cannot share a $2$-path with a vertex $v=(m,n,p)$ with $m+n+p\geq 8$ as $(412-233)$ and $(332-233)$ are reducible respectively by \Cref{reducible pair}(vi) and (vii). As a result, $u$ receives charge $\frac12$ by \ru2(ii). To sum up, we have $$\mu^*(u) = \frac{15}{2} + \frac12 - 3 - 3 - 2 = 0 $$

\item If $u$ is a $(5,2,1)$, then $u$ cannot share a $2$-path with a vertex $v=(l,i,j)$ with $l+i+j\geq 8$ and $u$ cannot share a $1$-path with a vertex $w=(m,n,p)$ with $m+n+p\geq 8$ at the same time, as $(412-233)$ and $(421-132-214)$ are reducible respectively by \Cref{reducible pair}(vi) and \Cref{reducible triple}(xiv). As a result, $u$ receives at least charge $\frac12$ by \ru1(ii) or \ru2(ii). To sum up, we have $$\mu^*(u) \geq \frac{15}{2} + \frac12 - 5 - 2 - 1 = 0 $$
\end{itemize}

\textbf{Case 2:} Suppose that $k+l+m\leq 7$ and that $u$ is a $(2^-,5^-,2^-)$.\\
First, observe that when $u$ is a $(2^-,2^-,2^-)$, it gives at most $\frac52$ along every incident path except for the case of \ru2(i), when it shares a 2-path with a $(2,2,5)$. Indeed, by \ru0, $u$ gives at most $\frac52$ to an adjacent 3-vertex. By \ru1 and \ru3, $u$ gives 1 to the 2-vertex on the 1-path and at most $\frac32$ to the other endvertex. By \ru2(ii), \ru2(iii), and \ru3, $u$ gives 2 to the 2-vertices on the 2-path and at most $\frac12$ to the other endvertex. As a result, $u$, a $(2^-,2^-,2^-)$ that does not share a 2-path with a $(2,2,5)$, verifies
$$ \mu^*(u) \geq \frac{15}{2} - 3\cdot \frac52 = 0$$
In other words, for the following values of $k,l,m$, we only need to look at $2\leq l\leq 5$. Moreover, when $l=2$, we can assume w.l.o.g. that the other endvertex of the 2-path is a $(2,2,5)$ since $k$, $l$, and $m$ are interchangeable.

Let $v=(i,j,k)$ share the $k$-path with $u$ and let $w=(m,n,p)$ share the $m$-path with $u$ (see \Cref{notation figure}). For each case, only \ru3, \ru{k} and \ru{m} apply, with the additional \ru2(i) when $l=2$.

\begin{figure}[htbp]
\begin{center}
\begin{tikzpicture}[scale=0.55]{thick}
\begin{scope}[every node/.style={circle,draw,minimum size=1pt,inner sep=2}]
	\node (0) at (-2,0) {};
    \node[fill] (1) at (0,0) {};
    \node[fill] (2) at (2,0) {};
    \node[fill,label={[label distance = +4pt]above right:$v$}] (3) at (4,0) {};
    \node[fill] (4) at (6,0) {};
    \node[fill] (5) at (8,0) {};
    \node[fill,label={[label distance = +4pt]above right:$u$}] (6) at (10,0) {};
    \node[fill] (7) at (12,0) {};
    \node[fill] (8) at (14,0) {};
    \node[fill,label={[label distance = +4pt]above right:$w$}] (9) at (16,0) {};
    \node[fill] (10) at (18,0) {};
    \node[fill] (11) at (20,0) {};
    \node (12) at (22,0) {};
    
    \node[fill] (3') at (4,2) {};
    \node[fill] (3'') at (4,4) {};
    \node (3''') at (4,6) {};
    
    \node[fill] (6') at (10,2) {};
    \node[fill] (6'') at (10,4) {};
    \node (6''') at (10,6) {};
    
    \node[fill] (9') at (16,2) {};
    \node[fill] (9'') at (16,4) {};
    \node (9''') at (16,6) {};
\end{scope}

\begin{scope}[every edge/.style={draw=black}]
    \path (0) edge (1);
    \path (1) edge[dashed] node[above] {$i$ vertices} (2);
    \path (2) edge (3);
    
    \path (3) edge (3');
    \path (3') edge[dashed] node[right] {$j$ vertices} (3'');
    \path (3'') edge (3''');
    
    \path (3) edge (4);
    \path (4) edge[dashed] node[above] {$k$ vertices} (5);
    \path (5) edge (6);
    
    \path (6) edge (6');
    \path (6') edge[dashed] node[right] {$l$ vertices} (6'');
    \path (6'') edge (6''');
    
    \path (6) edge (7);
    \path (7) edge[dashed] node[above] {$m$ vertices} (8);
    \path (8) edge (9);
    
    \path (9) edge (9');
    \path (9') edge[dashed] node[right] {$n$ vertices} (9'');
    \path (9'') edge (9''');
    
    \path (9) edge (10);
    \path (10) edge[dashed] node[above] {$p$ vertices} (11);
    \path (11) edge (12);
\end{scope}
\end{tikzpicture}
\caption{Notations.}
\label{notation figure}
\end{center}
\end{figure}

\begin{itemize}
\item If $u$ is a $(0,5^-,0)$, then we distinguish the two following cases:

If $2\leq l\leq 3$, then $u$ gives at most 3 along the $l$-path: either 3 to the 2-vertices in the case of a 3-path or 2 to the 2-vertices and $\frac34$ to the other endvertex by \ru2(i). Since $(550-020-045)$ is reducible by \Cref{reducible triple}(i), $u$ cannot give $\frac52$ twice to $v$ and $w$ by \ru0. So at worst, we have
$$\mu^*(u) = \frac{15}2 - 3 - \frac52 - \frac32 = \frac12 $$

If $4\leq l\leq 5$, then $u$ gives at most 5 to the 2-vertices along the $l$-path.
Since $(440-040-024)$ is reducible by \Cref{reducible triple}(ii), if $u$ gives at least $\frac32$ to $v$ by \ru0(i) or \ru0(ii), then $u$ does not give charge to $w$.

So at worst, we have
$$\mu^*(u) = \frac{15}2 - 5 - \frac52 = 0 $$

\item If $u$ is a $(0,5^-,1)$, then we distinguish the two following cases:

If $2\leq l\leq 3$, then $u$ gives at most 3 along the $l$-path: either 3 to the 2-vertices in the case of a 3-path or 2 to the 2-vertices and $\frac34$ to the other endvertex by \ru2(i). Since $(550-021-134)$ and $(420-031-134)$ are reducible respectively by \Cref{reducible triple}(iii), $u$ cannot give $\frac52$ twice to $v$ and $w$ by \ru0 and \ru1 ($1+\frac32$ in the case of \ru1(i)). So at worst, we have
$$\mu^*(u) = \frac{15}2 - 3 - \frac52 - \frac32 = \frac12 $$

If $4\leq l\leq 5$, then $u$ gives at most 5 along the $l$-path. 
\begin{itemize}
\item If $w$ is a $(1,3^+,4^+)$, then $v$ cannot be a $(4^+,2^+,0)$ since $(420-031-134)$ is reducible by \Cref{reducible triple}(iv). As a result, $u$ gives at most $\frac32$ along the 1-path by \ru1 and nothing to its adjacent 3-vertex by \ru0. So at worst, we have
$$\mu^*(u) = \frac{15}2 - 5 - 1 - \frac32 = 0 $$
\item If $w$ is not a $(1,3^+,4^+)$, then $u$ gives at most $\frac12$ along the 1-path by \ru1 and at most $\frac12$ to its adjacent 3-vertex by \ru0 since $(540-014)$ is reducible by \Cref{reducible pair}(ii). So at worst, we have
$$\mu^*(u) = \frac{15}2 - 5 - 1 - \frac12 - \frac12 = \frac12 $$
\end{itemize}

\item If $u$ is a $(0,5^-,2)$, then we distinguish the four following cases:

If $l=2$, then $u$ gives $2+\frac34$ along the 2-path by \ru3 and \ru2(i). Since $(550-022-224)$ is reducible by \Cref{reducible triple}(v), $v$ cannot be a $(5,5,0)$. As a result, $u$ gives at most $\frac32$ to its adjacent 3-vertex by \ru0 and $2+\frac34$ along each 2-path by \ru3 and \ru2. So at worst, we have
$$ \mu^*(u) = \frac{15}2 - \frac32 - 2\cdot \left(2+\frac34\right) = \frac12 $$

If $l = 3$, then $u$ gives 3 along the $l$-path and 2 along the 2-path by \ru3. 
\begin{itemize}
\item If $v$ is a $(5,4^+,0)$, then $w$ cannot be a $(2,1^+,4^+)$ nor a $(2,3,3)$ as $(540-032-214)$ and $(540-032-233)$ are reducible respectively by \Cref{reducible triple}(vi) and \Cref{reducible triple}(vii). As a result, $u$ gives at most $\frac52$ to $v$ by \ru0 and nothing to $w$ by \ru2. So at worst, we have
$$ \mu^*(u) = \frac{15}2 - 3 - 2 - \frac52 = 0 $$
\item If $v$ is not a $(5,4^+,0)$, then $u$ gives at most $\frac12$ to $v$ by \ru0 and at most $\frac34$ to $w$ by \ru2. So at worst, we have
$$ \mu^*(u) = \frac{15}2 - 3 - 2 - \frac12 - \frac34 = \frac54 $$
\end{itemize}

If $l=4$, then $u$ gives 4 along the $l$-path and $2$ along the 2-path by \ru3. Since $(430-024)$ is reducible by \Cref{reducible pair}(i), $v$ cannot be a $(4^+,3^+,0)$. As a result, $u$ gives at most $\frac14$ to $v$ by \ru0 and at most $\frac34$ to $w$ by \ru2. So at worst, we have
$$ \mu^*(u) = \frac{15}2 - 4 - 2 - \frac14 - \frac34 = \frac12 $$

If $l=5$, then $u$ gives 5 along the $l$-path and 2 along the 2-path by \ru3. 
\begin{itemize}
\item If $v$ is a $(4^+,2^+,0)$, then $w$ cannot be a $(2,1^+,4^+)$ nor a $(2,3,3)$ as $(420-042-214)$ and $(420-042-233)$ are reducible respectively by \Cref{reducible triple}(viii) and \Cref{reducible triple}(ix). Moreover, $v$ cannot be a $(4^+,3^+,0)$ since $(430-024)$ is reducible by \Cref{reducible pair}(i). As a result, $u$ gives at most $\frac14$ to $v$ by \ru0 and nothing to $w$ by \ru2. So at worst, we have
$$ \mu^*(u) = \frac{15}2 - 5 - 2 - \frac14 = \frac14 $$

\item If $v$ is not a $(4^+,2^+,0)$, then $v=(i,j,k)$ with $i+j+k\leq 7$. Thus, $u$ receives $\frac14$ from $v$ by \ru0(iv). Moreover, $u$ gives nothing to $v$ by \ru0 and at most $\frac34$ to $w$ by \ru2. So at worst, we have
$$ \mu^*(u) = \frac{15}2 - 5 - 2 + \frac14 - \frac34 = 0 $$  
\end{itemize}

\item If $u$ is a $(1,5^-,1)$, then we distinguish the three following cases:

If $l=2$, then $u$ gives $2+\frac34$ along the 2-path by \ru3 and \ru2(i) and 1 to each 2-vertex on the 1-paths by \ru3. Since $(431-112-224)$ is reducible by \Cref{reducible triple}(xii), $v$ cannot be a $(4^+,3^+,1)$. The same holds for $w$. As a result, $u$ gives at most $\frac12$ twice to $v$ and $w$ by \ru1. So at worst, we have
$$ \mu^*(u) = \frac{15}2 - 2 - \frac34 - 1 - 1 - 2\cdot\frac12 = \frac74 $$

If $l=3$, then $u$ gives 3 to the $l$-path and 1 to each 2-vertex on the 1-paths by \ru3. Since $(431-131-124)$ is reducible by \Cref{reducible triple}(x), $v$ and $w$ cannot both be $(4^+,3^+,1)$s. As a result, $u$ cannot give $\frac32$ twice by \ru1. So at worst, we have
$$ \mu^*(u) = \frac{15}2 - 3 - 1 - 1 - \frac32 - \frac12 = \frac12 $$ 

If $4\leq l\leq 5$, then $u$ gives at most 5 along the $l$-path, 1 to each 2-vertex on the 1-paths by \ru3. Since $(431-114)$ is reducible by \Cref{reducible pair}(iii), $u$ cannot give more than $\frac12$ to $v$ nor $w$ by \ru1. Moreover, since $(421-141-124)$ is also reducible by \Cref{reducible triple}(xi), $u$ cannot give $\frac12$ twice by \ru1. So at worst, we have
$$ \mu^*(u) = \frac{15}2 - 5 - 1 - 1 - \frac12 = 0 $$

\item If $u$ is a $(1,5^-,2)$, then $l\leq 4$ since $k+l+m\leq 7$. Thus, we distinguish the three following cases:

If $l=2$, then $u$ gives $2+\frac34$ along at least one of the 2-paths by \ru3 and \ru2(i) and 1 to each 2-vertex on the 1-path and other 2-path by \ru3. Since $(431-112-224)$ is reducible by \Cref{reducible triple}(xii), $v$ cannot be a $(4^+,3^+,1)$. As a result, $u$ gives at most $\frac12$ to $v$ by \ru1 and at most $\frac34$ to $w$ by \ru2. So at worst, we have
$$ \mu^*(u) = \frac{15}2 - 2 - \frac34 - 2 - 1 - \frac12 - \frac34 = \frac12 $$

If $l=3$, then $u$ gives 3 along the $l$-path and 1 to each 2-vertex on the 1-path and 2-path by \ru3.
\begin{itemize}
\item If $v$ is a $(4^+,2^+,1)$, then $w$ cannot be a $(2,1^+,4^+)$ nor a $(2,3,3)$ since $(421-132-233)$ and $(421-132-214)$ are reducible respectively by \Cref{reducible triple}(xiii) and (xiv). As a result, $u$ gives at most $\frac32$ to $v$ by \ru1 and nothing to $w$ by \ru2. So at worst, we have
$$ \mu^*(u) = \frac{15}2 - 3 - 2 - 1 - \frac32 = 0 $$
\item If $v$ is not a $(4^+,2^+,1)$, then $u$ gives nothing to $v$ by \ru1 and at most $\frac34$ to $w$ by \ru2. So at worst, we have
$$ \mu^*(u) = \frac{15}2 - 3 - 2 - 1 - \frac34 = \frac34 $$
\end{itemize}  

If $l=4$, then $u$ gives 4 along the $l$-path and 1 to each 2-vertex on the 1-path and 2-path by \ru3. Since $(431-114)$, $(422-214)$, and $(412-233)$ are reducible respectively by \Cref{reducible pair}(iii), (v) and (vi), $v$ cannot be a $(4^+,3^+,1)$ and $w$ cannot be a $(2,2^+,4^+)$ nor a $(2,3,3)$. Moreover, $(421-132-214)$ is reducible by \Cref{reducible triple}(xiv). As a result, $u$ can give at most $\frac12$ once to either $v$ or $w$. So at worst, we have
$$ \mu^*(u) = \frac{15}2 - 4 - 2 - 1 - \frac12 = 0 $$ 

\item If $u$ is a $(2,5^-,2)$, then $l\leq 3$, since $k+l+m\leq 7$. Thus, we distinguish the two following cases:

If $l=2$, then $u$ gives $2+\frac34$ along at least one of the 2-paths by \ru3 and \ru2(i) and 1 to each 2-vertex on the 2-paths by \ru3. Since $(422-222-214)$ and $(332-222-224)$ are reducible respectively by \Cref{reducible triple}(xv) and (xvi), $v$ cannot be a $(4^+,1^+,2)$ nor a $(3,3,2)$. The same holds for $w$. As a result, $u$ gives nothing to $v$ nor $w$ by \ru2. So at worst, we have
$$ \mu^*(u) = \frac{15}2 - 2 -\frac34 - 2 - 2 = \frac34 $$ 

If $l=3$, then $u$ gives 3 along the $l$-path and 1 to each 2-vertex on the 2-paths by \ru3. 
\begin{itemize}
\item If either $v$ or $w$ is a $(3,3,2)$, then the other cannot be a $(2,3,3)$ nor a $(2,1^+,4^+)$ as $(332-232-233)$ and $(332-232-214)$ are reducible respectively by \Cref{reducible triple}(xvii) and (xviii). So, $u$ gives only $\frac12$ once to either $v$ or $w$ by \ru2. So at worst, we have
$$ \mu^*(u) = \frac{15}2 - 3 - 2 - 2 -\frac12 = 0 $$
\item If neither $v$ nor $w$ is a $(3,3,2)$, then the remaining cases are as follows. Since $(422-223)$ is reducible by \Cref{reducible pair}(iv), $v$ cannot be a $(4^+,2^+,2)$. The same holds for $w$. Moreover, since $(412-232-214)$ is reducible by \Cref{reducible triple}(xix), they cannot both be $(4^+,1^+,2)$s. As a result, $u$ gives at most $\frac12$ once to either $v$ or $w$ by \ru2. So at worst, we have
$$ \mu^*(u) = \frac{15}2 - 3 - 2 - 2 -\frac12 = 0 $$

\end{itemize}
\end{itemize}

\textbf{Case 3:} Suppose that $k+l+m\leq 7$ and that $u$ is a $(3^+,5^-,3^+)$. \\
Since $k+l+m\leq 7$, the only possibilities for $u$ are as follows:
\begin{itemize}
\item If $u$ is a $(3,0,3)$, then $u$ can only give charge by \ru0 and \ru3. Since $(330-045)$ is reducible by \Cref{bi}(v), $u$ can give at most $\frac12$ to another 3-vertex by \ru0(iii) or \ru0(iv). As a result,
$$\mu^*(u)\geq \frac{15}{2} - \frac12 - 3 - 3 = 1 $$
\item If $u$ is a $(3,1,3)$, then $u$ can only give charge by \ru1 and \ru3. Since $(431-133)$ is reducible by \Cref{bi}(vi), $u$ can give at most $\frac12$ to another 3-vertex by \ru1(ii). As a result,
$$\mu^*(u)\geq \frac{15}{2} - \frac12 - 3 - 3 - 1 = 0 $$
\item If $u$ is a $(4,0,3)$, then $u$ can only give charge by \ru0 and \ru3. Since $(430-024)$ is reducible by \Cref{reducible pair}(i), $u$ actually does not give charge by \ru0. As a result,
$$\mu^*(u)\geq \frac{15}{2} - 4 - 3 = \frac12 $$ 
\end{itemize}

To conclude, we started with a charge assignment with a negative total sum, but after the discharging procedure, which preserved that sum, we end up with a non-negative one, which is a contradiction. In other words, there exists no counter-example $G$ to \Cref{main theorem}. 

\section{A non 4-colorable subcubic planar graph of girth 11}
\label{sec3}

In \citep{dvo08}, Dvo\u{r}\'ak, \u{S}krekovski, and Tancer presented a non $4$-colorable, planar, and subcubic graph with girth at least 9. The main building block of that graph relies upon an interesting property of $4$-colorings on path of length 5. Using the same property we managed to build a non $4$-colorable planar subcubic graph of girth 11. 

\begin{lemma}\label{4-path lemma}
Let $H$ be a subcubic graph of girth at least 11 and $\phi$ a $4$-coloring of $H$. Let $u_1u_2u_3u_4u_5u_6$ be a path of length $5$ in $H$, if $\phi(u_1)=\phi(u_6)$, then $\phi(u_2)=\phi(u_5)$.
\end{lemma}

\begin{proof}
Since $H$ has girth at least 11, all considered vertices are distinct. Suppose by contradiction that $\phi(u_1)=\phi(u_6)$ but $\phi(u_2)\neq\phi(u_5)$. W.l.o.g. we set $\phi(u_1)=\phi(u_6)=a$, $\phi(u_2)=b$, and $\phi(u_5)=c$. Since $u_3$ sees $u_1$, $u_2$, and $u_5$, colored respectively $a$, $b$, and $c$, it must be colored $d$. Finally, $u_4$ sees $u_2$, $u_3$, $u_5$, and $u_6$, colored respectively by $b$, $d$, $c$, and $a$. Thus, $u_4$ is non-colorable, which is a contradiction since $\phi$ is a $4$-coloring of $H$. 
\end{proof}

\begin{lemma} \label{triangle lemma}
Let $H$ be a subcubic graph of girth 11 and $\phi$ a $4$-coloring of $H$. Let $u_1u_2u_3u_4u_5u_6$, $u_3u'_1u'_2u'_3u'_4v_1$, $u_4u''_1u''_2u''_3u''_4v_1$  be paths of length $5$ in $H$. Let $v_0\notin \{u'_4,u''_4\}$ be adjacent to $v_1$. If $\phi(u_1)=\phi(u_6)=\phi(v_0)$, then $\phi(u_2)=\phi(u_5)=\phi(v_1)$. 
\end{lemma}

\begin{proof}
Since $H$ has girth 11, all considered vertices are distinct. We assume w.l.o.g. that $\phi(u_1)=\phi(u_6)=\phi(v_0)=a$. By \Cref{4-path lemma}, since $\phi(u_1)=\phi(u_6)$, we must have $\phi(u_2)=\phi(u_5)$. W.l.o.g. we set $\phi(u_2)=\phi(u_5)=b$. As a result, we have $\{\phi(u_3),\phi(u_4)\}=\{c,d\}$. We assume w.l.o.g. that $\phi(u_3)=c$ and $\phi(u_4)=d$. Now, suppose by contradiction that $\phi(v_1)=c$. By \Cref{4-path lemma}, since $\phi(u_3)=\phi(v_1)$, we must have $\phi(u'_1)=\phi(u'_4)=a$. However, this is impossible since $u'_4$ sees $v_0$ which is colored $a$. By symmetry, the same argument holds when $\phi(v_1)=d$. Finally, since $v_1$ also sees $v_0$, thus $\phi(v_1)\notin\{a,c,d\}$, and so $\phi(v_1)=b=\phi(u_2)=\phi(u_5)$.
\end{proof}

\begin{figure}[htbp]
\begin{center}
\begin{tikzpicture}[scale=0.7]{thick}
\begin{scope}[every node/.style={circle,draw,minimum size=1pt,inner sep=2}]
    \node[label={above:$u_1$},label={below:$a$}] (1) at (0,0) {};
    \node[label={above:$u_2$},label={below:$b$}] (2) at (2,0) {};
    \node[label={above:$u_3$},label={below:$d$}] (3) at (4,0) {};
    \node[label={above:$u_4$}] (4) at (6,0) {};
    \node[label={above:$u_5$},label={below:$c$}] (5) at (8,0) {};
    \node[label={above:$u_6$},label={below:$a$}] (6) at (10,0) {};
\end{scope}

\begin{scope}[every edge/.style={draw=black}]
    \path (1) edge (2);
    \path (2) edge (3);
    \path (3) edge (4);
    \path (4) edge (5);
    \path (5) edge (6);
\end{scope}
\end{tikzpicture}
\caption{A non-valid coloring of $H$ in \Cref{4-path lemma}.}
\end{center}
\end{figure}

\begin{figure}[htbp]
\begin{center}
\begin{tikzpicture}[scale=0.7]{thick}
\begin{scope}[every node/.style={circle,draw,minimum size=1pt,inner sep=2}]
    \node[label={above:$u_1$},label={below:$a$}] (1) at (0,0) {};
    \node[label={above:$u_2$},label={below:$b$}] (2) at (2,0) {};
    \node[fill, label={above left:$u_3$},label={below:$c$}] (3) at (4,0) {};
    \node[fill, label={above right:$u_4$},label={below:$d$}] (4) at (6,0) {};
    \node[label={above:$u_5$},label={below:$b$}] (5) at (8,0) {};
    \node[label={above:$u_6$},label={below:$a$}] (6) at (10,0) {};
    
    \node[label={above left:$u'_1$},label={below left:$a$}] (1') at (4,1.5) {};
    \node[label={above left:$u'_2$}] (2') at (4,2.5) {};
    \node[label={above left:$u'_3$}] (3') at (4,3.5) {};
    \node[label={above left:$u'_4$}] (4') at (4,4.5) {};
    
    \node[label={above right:$u''_1$},label={below right:$a$}] (1'') at (6,1.5) {};
    \node[label={above right:$u''_2$}] (2'') at (6,2.5) {};
    \node[label={above right:$u''_3$}] (3'') at (6,3.5) {};
    \node[label={above right:$u''_4$}] (4'') at (6,4.5) {};
    
    \node[fill,label={left:$v_1$},label={right:$c$}] (v1) at (5,5.5) {};
    \node[label={left:$v_0$},label={right:$a$}] (v0) at (5,6.5) {};
\end{scope}

\begin{scope}[every edge/.style={draw=black}]
    \path (1) edge (2);
    \path (2) edge (3);
    \path (3) edge (4);
    \path (4) edge (5);
    \path (5) edge (6);
    
    \path (3) edge (1');
    \path (1') edge (2');
    \path (2') edge (3');
    \path (3') edge (4');
    \path (4') edge (v1);
    
    \path (4) edge (1'');
    \path (1'') edge (2'');
    \path (2'') edge (3'');
    \path (3'') edge (4'');
    \path (4'') edge (v1);
    
    \path (v0) edge (v1);
\end{scope}
\end{tikzpicture}
\caption{A non-valid coloring of $H$ in \Cref{triangle lemma}.}
\end{center}
\end{figure}

\begin{lemma} \label{Gneq}
The graph $G_{\neq}(u,v)$ in \Cref{Gneq figure} has the following properties:
\begin{itemize}
\item $G_{\neq}(u,v)$ is planar and subcubic.
\item $G_{\neq}(u,v)$ has girth 11.
\item The distance in $G_{\neq}(u,v)$ between $u$ and $v$ is 7.
\item Every $4$-coloring $\phi$ of $G_{\neq}(u,v)$ satisfies $\phi(u)\neq \phi(v)$.
\end{itemize}
\end{lemma} 

\begin{proof}
One can verify that $G_{\neq}(u,v)$ is planar, subcubic, has girth 11, and that the distance between $u$ and $v$ is 7 thanks to \Cref{Gneq figure}. It remains to prove that $\phi(u)\neq \phi(v)$ for every $4$-coloring $\phi$ of $G_{\neq}(u,v)$. 

Suppose by contradiction that there exists a $4$-coloring $\phi$ such that $\phi(u)=\phi(v)=a$. We can assume w.l.o.g. that $\phi(u_1)=b$, $\phi(u_2)=c$, and $\phi(v_5)=d$. Since $u_6$ sees $v$ which is colored $a$, we distinguish the following cases based on $\phi(u_6)$:
\begin{itemize}
\item If $\phi(u_6)=b$, then $\phi(u_5)=\phi(u_2)=c$ by \Cref{4-path lemma} as $\phi(u_6)=\phi(u_1)$. As a result, $\phi(v_1)=d$. Since $v_2$ and $v_4$ both see $b$ and $d$, we have $\{\phi(v_2),\phi(v_4)\}=\{a,c\}$. Now, $v_3$ sees $\{\phi(v_1),\phi(v_2),\phi(v_4),\phi(v_5)\}=\{d,a,c\}$, so $\phi(v_3)=b$. Finally, $v_7$ sees $\{\phi(v_2),\phi(v_3),\phi(v_4)\}=\{a,b,c\}$, hence $\phi(v_7)=d$. However, this is impossible since $\phi(u_1)=\phi(u_6)=\phi(v_3)=b$, thus $\phi(u_2)=\phi(u_5)=\phi(v_7)=c$ by \Cref{triangle lemma}.
 
\item If $\phi(u_6)=c$, then we have the two following cases:
\begin{itemize}
\item If $\phi(v_1)=b$, then $\phi(v_2)=\phi(v_5)=d$ by \Cref{4-path lemma} as $\phi(v_1)=\phi(u_1)$. As a result, $\phi(u_5)=d$ and $\phi(v_6)=a$. Since $v_3$ and $v_4$ both see $b$ and $d$, we have $\{\phi(v_3),\phi(v_4)\}=\{a,c\}$. Now, $v_7$ sees $\{\phi(v_2),\phi(v_3),\phi(v_4)\}=\{d,a,c\}$, so $\phi(v_7)=b$. Since $u_3$ sees $b$, $c$, and $d$, $\phi(u_3)=a$ and consequently, $\phi(u_4)=b$ and $\phi(w_1)=c$. However, this is impossible since $\phi(u_4)=\phi(v_7)=\phi(v_1)=b$, thus $\phi(w_1)=\phi(w_4)=\phi(v_6)=a$ by \Cref{triangle lemma}.

\item If $\phi(v_1)=d$, then $\phi(u_5)=b$. All three vertices $v_2$, $v_3$, and $v_4$ see $d$, so $\{\phi(v_2),\phi(v_3),\phi(v_4)\}$ $=\{a,b,c\}$. As a result, $\phi(v_7)=d$. Both $u_3$ and $u_4$ see $b$ and $c$, so $\{\phi(u_3),\phi(u_4)\}=\{a,d\}$. Since $w_1$ sees $\{\phi(u_3),\phi(u_4),\phi(u_5)\}=\{a,d,b\}$, $\phi(w_1)=c$. Due to \Cref{triangle lemma}, we must have $\phi(u_4)=a$. Otherwise, by \Cref{triangle lemma}, $\phi(u_4)=d=\phi(v_7)=\phi(v_1)$ and $\phi(w_1)=\phi(w_4)=\phi(v_6)=c$ which is impossible since $v_6$ sees $u_6$ colored $c$. Thus, $\phi(u_3)=d$ and $\phi(t_1)=b$. However, this is also impossible since $\phi(u_3)=\phi(v_7)=\phi(v_5)=d$, thus $\phi(t_1)=\phi(t_4)=\phi(v_8)=b$ by \Cref{triangle lemma} and $v_8$ sees $u_1$ colored $b$.
\end{itemize}

\item If $\phi(u_6)=d$, then $\phi(v_1)=\phi(v_4)$ by \Cref{4-path lemma} as $\phi(u_6)=\phi(v_5)$. Since $v_4$ sees $b$ and $d$ and $v_1$ sees $a$ and $d$, $\phi(v_4)=\phi(v_1)=c$. As a result, $\phi(u_5)=b$ and $\phi(v_8)=a$. Both $v_2$ and $v_3$ see $c$ and $d$, so $\{\phi(v_2),\phi(v_3)\}=\{a,b\}$. Now, $v_7$ sees $\{\phi(v_2),\phi(v_3),\phi(v_4)\}=\{a,b,c\}$, so $\phi(v_7)=d$. Since $u_4$ sees $d$, $b$, and $c$, $\phi(u_4)=a$ and consequently, $\phi(u_3)=d$ and $\phi(t_1)=b$. However, this is impossible since $\phi(u_3)=\phi(v_7)=\phi(v_5)=d$, thus $\phi(t_1)=\phi(t_4)=\phi(v_8)=a$ by \Cref{triangle lemma}. 
\end{itemize}
\end{proof}

\begin{figure}[htbp]
\begin{center}
\subfigure[\label{Gneq figure}The gadget $G_{\neq}(u,v)$ in \Cref{Gneq}.]{
\begin{tikzpicture}[scale=0.7]{thick}
\begin{scope}[every node/.style={circle,draw,minimum size=1pt,inner sep=2}]
	\node[label={left:$u$}] (u) at (-1,0) {};
	\node[label={right:$v$}] (v) at (17,0) {};	
	\node[fill,label={above left:$u_1$}] (u1) at (0,0) {};
	\node[fill,label={above left:$u_2$}] (u2) at (1,4) {};
	\node[fill,label={above:$u_3$}] (u3) at (5,8) {};
	\node[fill,label={above:$u_4$}] (u4) at (11,8) {};
	\node[fill,label={above right:$u_5$}] (u5) at (15,4) {};
	\node[fill,label={above right:$u_6$}] (u6) at (16,0) {};
	
	\node[fill,label={below right:$v_1$}] (v1) at (15,-4) {};
	\node[fill,label={below right:$v_2$}] (v2) at (12,-7) {};
	\node[fill,label={below:$v_3$}] (v3) at (8,-8) {};
	\node[fill,label={below left:$v_4$}] (v4) at (4,-7) {};
	\node[fill,label={below left:$v_5$}] (v5) at (1,-4) {};
	
	\node[fill,label={right:$v_7$}] (v7) at (8,-6) {};
	
	\node[fill,label={right:$t_4$}] (t4) at (7.35,-3) {};
	\node[fill,label={right:$w_4$}] (w4) at (8.65,-3) {};
	\node[fill,label={right:$t_3$}] (t3) at (6.73,0) {};
	\node[fill,label={right:$w_3$}] (w3) at (9.27,0) {};
	\node[fill,label={right:$t_2$}] (t2) at (5.86,4) {};
	\node[fill,label={right:$w_2$}] (w2) at (10.14,4) {};
	\node[fill,label={right:$t_1$}] (t1) at (5.44,6) {};
	\node[fill,label={right:$w_1$}] (w1) at (10.56,6) {};
	
	\node[fill,label={left:$v_8$}] (v8) at (2,-3) {};
	
	\node[fill,label={below:$t''_4$}] (t''4) at (3,-2.37) {};
	\node[fill,label={below:$t''_3$}] (t''3) at (4,-1.74) {};
	\node[fill,label={below:$t''_2$}] (t''2) at (5,-1.1) {};
	\node[fill,label={below:$t''_1$}] (t''1) at (6,-0.48) {};
	
	\node[fill,label={left:$t'_4$}] (t'4) at (2.83,-1.5) {};
	\node[fill,label={left:$t'_3$}] (t'3) at (3.67,0) {};
	\node[fill,label={left:$t'_2$}] (t'2) at (4.5,1.5) {};
	\node[fill,label={left:$t'_1$}] (t'1) at (5.33,3) {};
	
	\node[fill,label={right:$v_6$}] (v6) at (14,-3) {};
	
	\node[fill,label={below:$w''_4$}] (w''4) at (13,-2.37) {};
	\node[fill,label={below:$w''_3$}] (w''3) at (12,-1.74) {};
	\node[fill,label={below:$w''_2$}] (w''2) at (11,-1.1) {};
	\node[fill,label={below:$w''_1$}] (w''1) at (10,-0.48) {};
	
	\node[fill,label={right:$w'_4$}] (w'4) at (13.17,-1.5) {};
	\node[fill,label={right:$w'_3$}] (w'3) at (12.33,0) {};
	\node[fill,label={right:$w'_2$}] (w'2) at (11.5,1.5) {};
	\node[fill,label={right:$w'_1$}] (w'1) at (10.67,3) {};
\end{scope}

\begin{scope}[every edge/.style={draw=black}]
    \path (u) edge (u1);
    \path (u6) edge (v);
    \path (u1) edge (u2);
    \path (u2) edge (u3);
    \path (u3) edge (u4);
    \path (u4) edge (u5);
    \path (u5) edge (u6);
    \path (u6) edge (v1);
    \path (v1) edge (v2);
    \path (v2) edge (v3);
    \path (v3) edge (v4);
    \path (v4) edge (v5);
    \path (v5) edge (u1);
    
    \path (v3) edge (v7);
    \path (v7) edge (u3);
    \path (v7) edge (u4);
    
    \path (v5) edge (v8);
    \path (v8) edge (t2);
    \path (v8) edge (t3);
    
    \path (v1) edge (v6);
    \path (v6) edge (w2);
    \path (v6) edge (w3);
\end{scope}
\end{tikzpicture}
}
\subfigure[Simplified drawing of $G_{\neq}(u,v)$.]{
\begin{tikzpicture}[scale=0.7]{thick}
\begin{scope}[every node/.style={circle,draw,minimum size=1pt,inner sep=2}]
	\node[label={above:$u$}] (u) at (0,0) {};
	\node[label={above:$v$}] (v) at (3,0) {};
	\node (g) at (1.5,0) {$G_{\neq}$};
\end{scope}

\begin{scope}[every edge/.style={draw=black}]
    \path (u) edge (g);
    \path (g) edge (v);
\end{scope}
\end{tikzpicture}
}
\caption{$G_{\neq}$.}
\end{center}
\end{figure}

\begin{lemma} \label{G'neq}
The graph $G'_{\neq}(u,v)$ in \Cref{G'neq figure} has the following properties:
\begin{itemize}
\item $G'_{\neq}(u,v)$ is planar and subcubic.
\item $G'_{\neq}(u,v)$ has girth 11.
\item The distance in $G'_{\neq}(u,v)$ between $u$ and $v$ is 10.
\item Every $4$-coloring $\phi$ of $G'_{\neq}(u,v)$ satisfies $\phi(u)\neq \phi(v)$.
\end{itemize}
\end{lemma}

\begin{figure}[htbp]
\begin{center}
\subfigure[\label{G'neq figure}The gadget $G'_{\neq}(u,v)$ in \Cref{G'neq}.]{
\begin{tikzpicture}[scale=0.7]{thick}
\begin{scope}[every node/.style={circle,draw,minimum size=1pt,inner sep=2}]
	\node[label={left:$u$}] (u) at (4,2) {};
	\node[label={above:$v$}] (v) at (14,2) {};
	
	\node[label={above:$w_1$}] (w1) at (12,2) {};
	\node[fill,label={above:$w_2$}] (w2) at (10,2) {};
	\node[label={above:$w_3$}] (w3) at (8,4) {};
	\node[label={above:$w_4$}] (w4) at (8,0) {};
	
	\node (g1) at (6,4) {$G_{\neq}$};
	\node (g2) at (6,0) {$G_{\neq}$};
\end{scope}

\begin{scope}[every edge/.style={draw=black}]
    \path (w2) edge (w3);
    \path (w2) edge (w4);
    \path (w3) edge (g1);
    \path (w4) edge (g2);
    \path (g1) edge (u);
    \path (g2) edge (u);
    \path (w1) edge (w2);
    \path (v) edge (w1);
\end{scope}
\end{tikzpicture}
}
\hfil
\subfigure[Simplified drawing of $G'_{\neq}(u,v)$.]{
\begin{tikzpicture}[scale=0.7]{thick}
\begin{scope}[every node/.style={circle,draw,minimum size=1pt,inner sep=2}]
	\node[label={above:$u$}] (u) at (0,0) {};
	\node[label={above:$v$}] (v) at (3,0) {};
	\node (g) at (1.5,0) {$G'_{\neq}$};
\end{scope}

\begin{scope}[every edge/.style={draw=black}]
    \path (u) edge[bend left=30] (g);
    \path (u) edge[bend right=30] (g);
    \path (g) edge (v);
\end{scope}
\end{tikzpicture}
}
\end{center}
\caption{$G'_{\neq}$.}
\end{figure}

\begin{proof}
One can verify that $G'_{\neq}(u,v)$ is planar, subcubic, has girth 11, and that the distance between $u$ and $v$ is 10 thanks to \Cref{G'neq figure} and \Cref{Gneq}. It remains to prove that $\phi(u)\neq \phi(v)$ for every $4$-coloring $\phi$ of $G'_{\neq}(u,v)$. Suppose by contradiction that there exists a 4-coloring $\phi$ of $G'_{\neq}(u,v)$ such that $\phi(u)=\phi(v)$, say $\phi(u)=a$. We only need to observe that $w_3$ and $w_4$ cannot be colored $a$ thanks to $G_{\neq}(u,v)$ and $w_1$ and $w_2$ cannot be colored $a$ since they see $v$. This is a contradiction as we have four vertices at distance two pairwise but only three colors left. 
\end{proof}

\begin{lemma} \label{Geq}
The graph $G_=(u,v)$ in \Cref{Geq figure} has the following properties:
\begin{itemize}
\item $G_=(u,v)$ is planar and subcubic.
\item $G_=(u,v)$ has girth 11.
\item The distance in $G_=(u,v)$ between $u$ and $v$ is 3.
\item Every $4$-coloring $\phi$ of $G_=(u,v)$ satisfies $\phi(u)=\phi(v)$.
\end{itemize}
\end{lemma}

\begin{figure}[htbp]
\begin{center}
\subfigure[\label{Geq figure}The gadget $G_=(u,v)$ in \Cref{Geq}.]{
\begin{tikzpicture}[scale=0.7]{thick}
\begin{scope}[every node/.style={circle,draw,minimum size=1pt,inner sep=2}]
	\node[label={above:$u$}] (u) at (0,2) {};
	\node[fill,label={above:$t_1$}] (t1) at (2,2) {};	
	\node[fill,label={above left:$t_2$}] (t2) at (4,2) {};
	\node[label={above:$v$}] (v) at (14,2) {};
	
	\node[fill,label={above:$w_1$}] (w1) at (12,2) {};
	\node[fill,label={above:$w_2$}] (w2) at (10,2) {};
	\node[label={above:$w_3$}] (w3) at (8,4) {};
	\node[label={above:$w_4$}] (w4) at (8,0) {};
	
	\node (g1) at (6,4) {$G_{\neq}$};
	\node (g2) at (6,0) {$G_{\neq}$};
\end{scope}

\begin{scope}[every edge/.style={draw=black}]
    \path (w2) edge (w3);
    \path (w2) edge (w4);
    \path (w3) edge (g1);
    \path (w4) edge (g2);
    \path (g1) edge (t2);
    \path (g2) edge (t2);
    \path (w1) edge (w2);
    \path (v) edge (w1);
    \path (u) edge (t2);
    \path (t1) edge[bend right=90] (w1);
\end{scope}
\end{tikzpicture}
}
\hfil
\subfigure[Simplified drawing of $G_=(u,v)$.]{
\begin{tikzpicture}[scale=0.7]{thick}
\begin{scope}[every node/.style={circle,draw,minimum size=1pt,inner sep=2}]
	\node[label={above:$u$}] (u) at (0,0) {};
	\node[label={above:$v$}] (v) at (3,0) {};
	\node (g) at (1.5,0) {$G_=$};
\end{scope}

\begin{scope}[every edge/.style={draw=black}]
    \path (u) edge (g);
    \path (g) edge (v);
\end{scope}
\end{tikzpicture}
}
\caption{$G_=$.}
\end{center}
\end{figure}

\begin{proof}
One can verify that $G_=(u,v)$ is planar, subcubic, has girth 11, and that the distance between $u$ and $v$ is 3 thanks to \Cref{Geq figure} and \Cref{Gneq}. It remains to prove that $\phi(u)=\phi(v)$ for every $4$-coloring $\phi$ of $G_=(u,v)$. Let $\phi$ be a $4$-coloring of $G_=(u,v)$, we can assume w.l.o.g. that $\phi(u)=a$, $\phi(t_1)=b$, $\phi(t_2)=c$, and $\phi(w_1)=d$. Observe that $v$ sees $t_1$ and $w_1$ colored respectively $b$ and $d$. Moreover, due to \Cref{G'neq}, $\phi(v)\neq \phi(t_2)=c$ as $G_=(u,v)$ contains $G'_{\neq}(t_2,v)$. As a result, we must have $\phi(v)=a=\phi(u)$.
\end{proof}

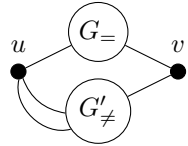
\begin{figure}[htbp]
\begin{center}
\begin{tikzpicture}[scale=0.7]{thick}
\begin{scope}[every node/.style={circle,draw,minimum size=1pt,inner sep=2}]
	\node[fill,label={above:$u$}] (u) at (0,0.75) {};
	\node[fill,label={above:$v$}] (v) at (3,0.75) {};
	\node (g1) at (1.5,0) {$G'_{\neq}$};
	\node (g2) at (1.5,1.5) {$G_=$};
\end{scope}

\begin{scope}[every edge/.style={draw=black}]
    \path (u) edge[bend right=30] (g1);
    \path (u) edge[bend right=60] (g1);
    \path (g1) edge (v);
    \path (u) edge (g2);
    \path (g2) edge (v);
\end{scope}
\end{tikzpicture}
\caption{\label{G figure}A non-4-colorable planar subcubic graph of girth 11.}
\end{center}
\end{figure}

As a direct consequence of \Cref{G'neq} and \Cref{Geq}, we get the following lemma.

\begin{lemma} \label{G}
The graph $G$ in \Cref{G figure} is a planar subcubic graph of girth 11 with $\chi^2(G)\geq 5$.
\end{lemma} 

In \citep{dvo08}, the authors also proved the NP-completeness of the problem of deciding if a planar subcubic graph of girth 9 is 4-colorable using a gadget that can reproduce colors at a far enough distance to preserve the girth condition. The same proof can be adapted directly to prove the NP-completeness of deciding if a planar subcubic graph of girth 11 is 4-colorable by using a concatenation of $G_=(u,v)$ to get a large enough distance.

\bibliographystyle{plainnat}
\bibliography{References}

\begin{thebibliography}{29}
\providecommand{\natexlab}[1]{#1}
\providecommand{\url}[1]{\texttt{#1}}
\expandafter\ifx\csname urlstyle\endcsname\relax
  \providecommand{\doi}[1]{doi: #1}\else
  \providecommand{\doi}{doi: \begingroup \urlstyle{rm}\Url}\fi

\bibitem[Bonamy et~al.(2014{\natexlab{a}})Bonamy, L{\'e}v{\^e}que, and
  Pinlou]{bon13}
M.~Bonamy, B.~L{\'e}v{\^e}que, and A.~Pinlou.
\newblock 2-distance coloring of sparse graphs.
\newblock \emph{Journal of Graph Theory}, 77\penalty0 (3), 2014{\natexlab{a}}.

\bibitem[Bonamy et~al.(2014{\natexlab{b}})Bonamy, L{\'e}v{\^e}que, and
  Pinlou]{bon14}
M.~Bonamy, B.~L{\'e}v{\^e}que, and A.~Pinlou.
\newblock Graphs with maximum degree {$\Delta \geq 17$} and maximum average
  degree less than 3 are list 2-distance ({$\Delta + 2$})-colorable.
\newblock \emph{Discrete Mathematics}, 317:\penalty0 19--32,
  2014{\natexlab{b}}.

\bibitem[Bonamy et~al.(2019)Bonamy, Cranston, and Postle]{bon19}
M.~Bonamy, D.~Cranston, and L.~Postle.
\newblock Planar graphs of girth at least five are square
  ({$\Delta+2$})-choosable.
\newblock \emph{Journal of Combinatorial Theory, Series B}, 134:\penalty0
  218--238, 2019.

\bibitem[Borodin and Ivanova(2011)]{bor11}
O.V. Borodin and A.O. Ivanova.
\newblock 2-distance 4-coloring of planar subcubic graphs.
\newblock \emph{Journal of Applied and Industrial Mathematics}, 5:\penalty0
  535--541, 2011.

\bibitem[Borodin and Ivanova(2012{\natexlab{a}})]{bi12}
O.V. Borodin and A.O. Ivanova.
\newblock List 2-facial 5-colorability of plane graphs with girth at least 12.
\newblock \emph{Discrete Mathematics}, 312:\penalty0 306--314,
  2012{\natexlab{a}}.

\bibitem[Borodin and Ivanova(2012{\natexlab{b}})]{bi12bis}
O.V. Borodin and A.O. Ivanova.
\newblock 2-distance 4-colorability of planar subcubic graphs with girth at
  least 22.
\newblock \emph{Discussiones Mathematicae Graph Theory}, 32\penalty0
  (1):\penalty0 141--151, 2012{\natexlab{b}}.

\bibitem[Borodin et~al.(2004)Borodin, Glebov, Ivanova, Neutroeva, and
  Tashkinov]{bor04}
O.V. Borodin, A.N. Glebov, A.O. Ivanova, T.K. Neutroeva, and V.A. Tashkinov.
\newblock Sufficient conditions for planar graphs to be $2$-distance
  {$(\Delta+1)$}-colorable.
\newblock \emph{Sibirskie Elektronnye Matematicheskie Izvestiya}, 1:\penalty0
  129--141, 2004.

\bibitem[Brown(1966)]{brown1966graphs}
William~G Brown.
\newblock On graphs that do not contain a thomsen graph.
\newblock \emph{Canadian Mathematical Bulletin}, 9\penalty0 (3):\penalty0
  281--285, 1966.

\bibitem[Bu and Shang(2016)]{bu16}
Y.~Bu and C.~Shang.
\newblock List 2-distance coloring of planar graphs without short cycles.
\newblock \emph{Discrete Mathematics, Algorithms and Applications}, 8\penalty0
  (1):\penalty0 1650013, 2016.

\bibitem[Bu and Zhu(2018)]{bu18b}
Y.~Bu and J.~Zhu.
\newblock \emph{Channel Assignment with r-Dynamic Coloring: 12th International
  Conference, AAIM 2018, Dallas, TX, USA, December 3–4, 2018, Proceedings},
  pages 36--48.
\newblock 2018.

\bibitem[Bu and Zhu(2012)]{bu11}
Y.~Bu and X.~Zhu.
\newblock An optimal square coloring of planar graphs.
\newblock \emph{Journal of Combinatorial Optimization}, 24:\penalty0 580--592,
  2012.

\bibitem[Bu et~al.(2015)Bu, Lv, and Yan]{bu15}
Y.~Bu, X.~Lv, and X.~Yan.
\newblock The list 2-distance coloring of a graph with {$\Delta(G)=5$}.
\newblock \emph{Discrete Mathematics, Algorithms and Applications}, 7\penalty0
  (2):\penalty0 1550017, 2015.

\bibitem[Cranston and Kim(2008)]{cra07}
D.~Cranston and S.-J. Kim.
\newblock List-coloring the square of a subcubic graph.
\newblock \emph{Journal of Graph Theory}, 1:\penalty0 65--87, 2008.

\bibitem[Cranston et~al.(2014)Cranston, Erman, and {\v S}krekovski]{cra13}
D.~Cranston, R.~Erman, and R.~{\v S}krekovski.
\newblock Choosability of the square of a planar graph with maximum degree
  four.
\newblock \emph{Australian Journal of Combinatorics}, 59\penalty0 (1):\penalty0
  86--97, 2014.

\bibitem[Dong and Lin(2016)]{dl16}
W.~Dong and W.~Lin.
\newblock An improved bound on 2-distance coloring plane graphs with girth 5.
\newblock \emph{Journal of Combinatorial Optimization}, 32\penalty0
  (2):\penalty0 645--655, 2016.

\bibitem[Dong and Lin(2017)]{don17}
W.~Dong and W.~Lin.
\newblock On 2-distance coloring of plane graphs with girth 5.
\newblock \emph{Discrete Applied Mathematics}, 217:\penalty0 495--505, 2017.

\bibitem[Dong and Xu(2017)]{don17b}
W.~Dong and B.~Xu.
\newblock 2-distance coloring of planar graphs with girth 5.
\newblock \emph{Journal of Combinatorial Optimization}, 34:\penalty0
  1302--1322, 2017.

\bibitem[Dvo{řá}k et~al.(2008{\natexlab{a}})Dvo{řá}k, Kr{à}l, Nejedl{\`y},
  and {Š}krekovski]{dvo08b}
Z.~Dvo{řá}k, D.~Kr{à}l, P.~Nejedl{\`y}, and R.~{Š}krekovski.
\newblock Coloring squares of planar graphs with girth six.
\newblock \emph{European Journal of Combinatorics}, 29\penalty0 (4):\penalty0
  838--849, 2008{\natexlab{a}}.

\bibitem[Dvo{řá}k et~al.(2008{\natexlab{b}})Dvo{řá}k, {Š}krekovski, and
  Tancer]{dvo08}
Z.~Dvo{řá}k, R.~{Š}krekovski, and M.~Tancer.
\newblock List-coloring squares of sparse subcubic graphs.
\newblock \emph{SIAM Journal on Discrete Mathematics}, 22\penalty0
  (1):\penalty0 139--159, 2008{\natexlab{b}}.

\bibitem[Hartke et~al.(2018)Hartke, Jahanbekam, and Thomas]{har16}
S.G. Hartke, S.~Jahanbekam, and B.~Thomas.
\newblock The chromatic number of the square of subcubic planar graphs.
\newblock arXiv:\penalty0 1604.06504, 2018.

\bibitem[Havet et~al.(2017)Havet, {V}an~{D}en Heuvel, McDiarmid, and
  Reed]{havet}
F.~Havet, J.~{V}an~{D}en Heuvel, C.~McDiarmid, and B.~Reed.
\newblock List colouring squares of planar graphs.
\newblock arXiv:\penalty0 0807.3233, 2017.

\bibitem[Ivanova(2011)]{iva11}
A.O. Ivanova.
\newblock List 2-distance ({$\Delta$}+1)-coloring of planar graphs with girth
  at least 7.
\newblock \emph{Journal of Applied and Industrial Mathematics}, 5\penalty0
  (2):\penalty0 221--230, 2011.

\bibitem[Kramer and Kramer(1969{\natexlab{a}})]{kramer1}
F.~Kramer and H.~Kramer.
\newblock {U}n problème de coloration des sommets d'un graphe.
\newblock \emph{Comptes Rendus Mathématique Académie des Sciences, Paris.},
  268:\penalty0 46--48, 1969{\natexlab{a}}.

\bibitem[Kramer and Kramer(1969{\natexlab{b}})]{kramer2}
F.~Kramer and H.~Kramer.
\newblock {E}in {F}ärbungsproblem der {K}notenpunkte eines {G}raphen
  bezüglich der {D}istanz p.
\newblock \emph{Revue Roumaine de Mathématiques Pures et Appliquées},
  14\penalty0 (2):\penalty0 1031--1038, 1969{\natexlab{b}}.

\bibitem[La and Montassier(2021)]{lm21}
H.~La and M.~Montassier.
\newblock 2-distance {$(\Delta+1)$}-coloring of sparse graphs using the
  potential method.
\newblock arXiv:\penalty0 2103.11687, 2021.

\bibitem[La et~al.(2021)La, Montassier, Pinlou, and Valicov]{lmpv19}
H.~La, M.~Montassier, A.~Pinlou, and P.~Valicov.
\newblock {$r$}-hued {$(r+1)$}-coloring of planar graphs with girth at least 8
  for {$r\geq 9$}.
\newblock \emph{European Journal of Combinatorics}, 91, 2021.

\bibitem[Lih et~al.(2003)Lih, Wang, and Zhu]{lwz03}
K.-W. Lih, W.-F. Wang, and X.~Zhu.
\newblock Coloring the square of a {$K_4$}-minor free graph.
\newblock \emph{Discrete Mathematics}, 269\penalty0 (1):\penalty0 303 -- 309,
  2003.

\bibitem[Thomassen(2018)]{tho18}
C.~Thomassen.
\newblock The square of a planar cubic graph is 7-colorable.
\newblock \emph{Journal of Combinatorial Theory, Series B}, 128:\penalty0
  192--218, 2018.

\bibitem[Wegner(1977)]{wegner}
G.~Wegner.
\newblock Graphs with given diameter and a coloring problem.
\newblock \emph{Technical report, {U}niversity of {D}ormund}, 1977.

\end{thebibliography}

\end{document}